\documentclass[12pt]{amsart}
\usepackage[hyphens]{url} 
\usepackage[colorlinks,allcolors=blue]{hyperref}
\usepackage{amssymb}
\usepackage{amsthm}
\usepackage{amsfonts}
\usepackage{mathtools}
\usepackage{amsmath}
\usepackage[all,cmtip]{xy}
\usepackage{fullpage}
\usepackage{mathrsfs}
\usepackage{graphicx}
\usepackage{tensor}
\usepackage{tikz}
\usepackage{tikz-cd}
\usepackage{quiver}
\usetikzlibrary{matrix}
\usetikzlibrary{arrows}
\usepackage{float}

\usepackage{kimballstyle}

\usepackage[style=alphabetic]{biblatex}
\tikzset{
    labl/.style={anchor=south, rotate=90, inner sep=.5mm}
}

\usepackage[textsize = tiny]{todonotes}

\begin{document}
\title{An Enriched Approach to the Strictification of $(\infty,1)$-categories}
\author{Kimball Strong}
\address{Department of Mathematics, Cornell University, Ithaca, NY 14853}
\email{ks2424@cornell.edu}
\date{}
\begin{abstract}
We define a functor which takes in an $(\infty,1)$-category and outputs an $(\omega,1)$-category, the natural maximally ``strict'' version of an $(\infty,1)$-category. We do this by modeling $(\infty,1)$-categories as categories enriched in $\infty$-groupoids, and then ``locally strictifying'' (applying the strictification of $\infty$-groupoids to each mapping space) to obtain a category enriched in $\omega$-groupoids with respect to the Gray tensor product, followed by ``globally strictifying'' (strictifying the enrichment from the Gray tensor product to the cartesian product) to obtain a category cartesian-enriched in $\omega$-groupoids, which is equivalently an $(\omega,1)$-category. We prove that this functor is conservative, generalizing the Homological Whitehead Theorem. 
\end{abstract}

\maketitle
\tableofcontents
\section*{Introduction}
The singular homology of a space is one of the most useful invariants of algebraic topology: it is strong enough to distinguish many spaces, yet frequently very computable. 
Part of its utility comes from the Homological Whitehead Theorem, which in its generalization to non-simply connected spaces states:
\begin{theorem}[Generalized Homological Whitehead Theorem]\label{thm:generalized-whitehead-groups}
Let $f: X \to Y$ be a map of spaces. If
\begin{itemize}
	\item $f$ induces an isomorphism on $\pi_0$ , and
	\item for any choice of basepoint $x_0$, $f$ induces an isomorphism $\pi_1(X,x_0) \to \pi_1(Y,f(x_0))$, and
	\item for any choice of basepoint, the induced map of universal covers $\widehat{f}: \widehat{X}_{x_0} \to \widehat{Y}_{f(x_0)}$ induces an isomorphism on homology,
\end{itemize}
then $f$ is a weak equivalence.
\end{theorem}
Originally due to Whitehead, this theorem was reinterpreted due to work of Brown and Higgins: they defined objects they called $\infty$-groupoids, later called strict $\infty$-groupoids, or as we shall refer to them \textit{$\omega$-groupoids} \cite{Brown_1981}.
As the name suggests, $\omega$-groupoids are a ``strict'' version of Grothendieck's $\infty$-groupoids, in the sense that whereas $\infty$-groupoids are to be thought as infinite dimensional groupoids in which all algebraic laws hold only up to coherent higher homotopy, $\omega$-groupoids are infinite dimensional groupoids in which all algebraic laws hold up to equality, e.g. composition of $n$-morphisms is genuinely associative.
$\omega$-groupoids have a natural notion of homotopy group, analogous to the automorphism group of a groupoid. 
This gives them a natural notion of weak equivalence, which can be extended to the structure of a model category \cite{BrownGolasinski1989}. 
Brown and Higgins define a functor $\Pi_\omega: \Top \to \stinfty$, called the ``fundamental $\omega$-groupoid'' functor, and show that for a space $X$ there are natural isomorphisms
$$\pi_0(\Pi_\omega(X)) \cong \pi_0(X) \quad \pi_1(\Pi_\omega(X), x_0) \cong \pi_1(X,x_0) \quad \pi_n(\Pi_\omega(X),x_0) \cong \pi_n(\widehat{X}_{x_0}, x_0)$$
where again $\widehat{X}_{x_0}$ is the universal cover of $X$ at $x_0$. From this it follows that Theorem \ref{thm:generalized-whitehead-groups} can be equivalently phrased as 
\begin{theorem}[{\cite{Brown_1981} \cite{Whitehead1948}}]\label{thm:strictification-infinity-groupoids-conservative}
The functor $\Pi_\omega: \Top \to \stinfty$ is conservative; i.e. reflects weak equivalences.
\end{theorem}
By Grothendieck's Homotopy Hypothesis, which states that $\infty$-groupoids have a homotopy theory equivalent to topological spaces, we think of $\Pi_\omega$ as presenting a functor $\st_0 : \infty\text{Gpd} \to \stinfty$, and interpret it as the ``strictification'' of an $\infty$-groupoid. 
This functor is a left adjoint to a natural inclusion $\iota_0: \stinfty \hookrightarrow \infty\text{Gpd}$.
So, another rephrasing of the same theorem is that \emph{strictification of $\infty$-groupoids is conservative}. 
\par $\infty$-groupoids are a special case (the $n=0$ case) of objects called $(\infty,n)$-categories, which especially in the case $n=1$ have become of primary importance in doing homotopical mathematics.
Similarly,
$\omega$-groupoids are themselves a special case of objects called $(\omega,n)$-categories, which can be thought of as strict versions of $(\infty,n)$-categories (Definition \ref{dfn:omega-n-cats}). 
We might expect for any $n$ to have an adjunction $\st_n: (\infty,n)\text{Cat}  \rightleftarrows \omegancat{n}: \iota_n$. 
Optimistically, we might hope that just as the functor $\st_0$ captures a large amount of information about spaces despite $\omega$-groupoids being inherently algebraic, $\st_n$ may also manage to provide a relatively strong invariant for $(\infty,n)$-categories despite $(\omega,n)$-categories being both fully algebraic and free from the subtleties of homotopy coherence that make $(\infty,n)$-categories so tricky to understand. 
\par In this paper we will verify a piece of this optimism by giving an enriched construction of $\text{St}_1$ and proving that it is conservative:
\begin{alphatheorem}[Theorem \ref{thm:infinity-one-conservativity-precisely}]\label{alphathm:infinity-one-conservativity}
There is a functor $\st_1: (\infty,1)\text{Cat} \to (\omega,1)\text{Cat}$, left adjoint to a natural inclusion functor. If $\C$ and $\D$ are $(\infty,1)$-categories, then
then a functor $F: \C \to \D$ of $(\infty,1)$-categories is an equivalence if and only if the strictification $\st_1(F): \st_1\C \to \st_1\D$ is an equivalence.
\end{alphatheorem}
By ``equivalence,'' we mean that both categories are model categories, with a natural notion of weak equivalence.
In light of the comparison with the $(\infty, 0)$ case, we regard this as a sort of ``Homological Whitehead Theorem for $(\infty,1)$ categories.'' 
It is not hard to show that Theorem \ref{alphathm:infinity-one-conservativity} immediately implies the verison for $\infty$-groupoids, Theorem \ref{thm:strictification-infinity-groupoids-conservative}.
\par Our approach is based on splitting the functor $\st_1$ into two parts as follows: 
$(\infty,1)$-categories can be modeled by categories weakly enriched in $\infty$-groupoids.
By strictifying every mapping $\infty$-groupoid, we obtain a category enriched in $\omega$-groupoids. 
However, this enrichment is with respect to a ``weak'' product (the Gray tensor product), and so we must further strictify the \textit{enrichment} to obtain a category enriched in $\omega$-groupoids (with respect to the cartesian product), which is equivalently an $(\omega,1)$-category.
If we think of strictification as a ``linearization,'' then this composite
\begin{center}
\begin{tikzcd}
(\infty,1)\cat \cong \infty\gpdcat \ar[r, "\leftone"] 
	& \stinfty^{\text{weak}}\text{-}\cat \ar[r, "\lefttwo"] & \stinfty^{\text{strict}}\text{-}\cat \cong (\omega,1)\cat
\end{tikzcd}
\end{center}
we think of as first ``locally linearizing'' followed by ``globally linearizing'' 
(this is of course somewhat vaguely stated; the particular models we work with are summarized in the top row of diagram \eqref{diagram:enriched-functors}).
This interpretation of strictification as a ``linearization'' is inspired by the tight connection between $\omega$-groupoids and chain complexes: any space with the homotopy type of an $\omega$-groupoid has universal cover a product of Eilenberg-Maclane spaces, equivalently the geometric realization of a chain complex \cite{Ara_2013}. 
It is not known whether this precisely characterizes such spaces, but Ehlers and Porter showed that the homotopy types modeled by $\omega$-groupoids can be characterized as variety within simplicial groupoids by the vanishing of words connected to the Whitehead bracket, the fundamental non-linear data of unstable homotopy theory \cite{EhlersPorter1997}.\footnote{A great deal of work on this has also been done in \cite{Carrasco1991} by Carrasco and Cegarra, who constructed gadgets called ``hypercrossed complexes'' which (roughly) generalize $\omega$-groupoids by equipping them with various dimension-increasing binary operations, which give a sort of non-linear enhancement to an $\omega$-groupoid. 
They show that these have a homotopy theory equivalent to the homotopy theory of $\infty$-groupoids, which characterizes the homotopy types modeled by $\omega$-groupoids as exactly those which can be modeled by a hypercrossed complex in which the binary operations are trivial, i.e. the ``linear'' objects.}
Furthermore, the explicit description of $\lefttwo$ in Appendix \ref{appendix:global-strictification-description} can be summarized as ``quotienting out by all non-linear (i.e. quadratic and higher) terms,'' and the chain complex version $\dgbarconstruction$ of $\righttwo$ recovers in a degenerate case with the functor of indecomposables of an augmented DGA. 
\par
To prove that this composite functor $\lefttwo \circ \leftone$ is conservative, it suffices to prove that each individual functor is conservative.
The first functor $\leftone$ is conservative because strictification of $\infty$-groupoids is conservative, and equivalences of $(\infty,1)$-categories are detected locally.
The second functor $\lefttwo$ is more difficult to understand; the major difficulty is its nonlocal nature: the hom $\omega$-groupoid $\lefttwo(\C)[x,y]$ depends not just on $\C[x,y]$ but on every composition map $\C[x,z] \otimes \C[z,y] \to \C[x,y]$. 
\par Concretely, the adjunction which $\lefttwo$ is the left adjoint of is a change of enrichment adjunction along an adjunction with lax monoidal right adjoint: namely, the identity adjunction $\Id: \stinfty^\otimes \rightleftarrows \stinfty^\times: \Id$ between the Gray monoidal structure and the cartesian monoidal structure on $\stinfty$. 
This adjunction is not strong monoidal, nor even Quillen monoidal, the homotopical analog.
We therefore have few general tools with which to analyze the induced change of enrichment adjunction;
one of the main purposes of the paper \cite{Strong2026} was to prove that this adjunction is Quillen.
Our proof of Theorem \ref{alphathm:infinity-one-conservativity} proceeds by analyzing a similar adjunction involving categories enriched in chain complexes over a varying groupoid of operators, which is constructed in \cite{Strong2026}.
\par We would like to note that Simon Henry and Felix Loubaton have constructed a different sort of strictification functor for $(\infty,n)$-categories and conjectured their functor is conservative \cite{HenryLoubaton2025}. Their approach contrasts with ours in that they do not use an enriched approach, and they do not produce the strictest possible version of $(\infty,n)$-categories, in which all axioms hold strictly; instead they produce a version in which all axioms are strict except invertibility. 
\par We also note that Chanavat and Hadzihasanovic have recently constructed what they call a ``semi-strictification'' of $(\infty,n)$-categories. Their meaning by ``strictification'' is of a different sort than either our meaning or Loubaton and Henry. Their goal is to find algebraic models for $(\infty,n)$-categories which are equivalent to the fully weak varieties (in particular, for $n=1$ equivalent to $(\infty,1)$-categories) but have some amount of strictness \cite{Chanavat2025}. 
By contrast, our meaning of ``strictification'' is a destructive operation intended to give a somewhat computable algebraic invariant.
\par Finally, we note that Burke in \cite{Burke2021} considers Quillen adjunctions between simplicial monoids, monoids in $\stinfty$ with respect to the Gray tensor product, and monoids in $\stinfty$ with respect to the cartesian product. Since monoids are equivalent to one-object categories, our work reproduces some of the work there.
\subsection*{Conventions}
We will consider a single category with two different closed symmetric monoidal structures; the category of crossed complexes $\crcom$ (Definition \ref{dfn:crossed-complex}) with both its categorical product $\times$ as well as its tensor product $\otimes$ (Definition \ref{dfn:tensor-product-of-crossed-complexes}). 
To differentiate these monoidal categories, we will notate them as $\crcom^\times$ and $\crcom^\otimes$, and use simply $\crcom$ when discussing results that either hold for both monoidal categories or results that are independent of the monoidal structure. 
We shall do similarly with the category $\ch$ of chain complexes over a varying groupoid with its cartesian and tensor monoidal structures (though there there will also be a further distinction between augmented and non-augmented).
\par
For $\V$ any of the symmetric monoidal categories $\{\sset, \crcom, \gpd\}$, we shall denote its unit by $\bullet$. 
For $\V$ a symmetric monoidal category and $\C$ a $\V$-category, we denote by $\C[x,y]$ the corresponding hom-object. The following construction occures frequently in analyzing enriched categories, particularly cellular objects in $\vcat$:
\begin{definition}
	Let $\V$ be a closed symmetric monoidal category. Denote by $\sus: \V \to \vcat$ the functor which takes an object $X \in \V$ to the $\V$-category with object set $\{0,1\}$ and hom-objects given by
	$$
		\sus X[x,y] =
		\begin{cases}
			\bullet & \text{ if } x=y \\
			X & \text{ if } x=0, \; y=1 \\
			\emptyset & \text{ if } x=1, \; y = 0
		\end{cases}	
	$$
	Where $\emptyset$ is the initial object of $\V$.
\end{definition}
In all of the model categories we will consider, we will write composition in the standard order, e.g. if we have morphisms \begin{tikzcd}  X \ar[r, "f"] & Y \ar[r, "g"] & Z \end{tikzcd} we write the composition $gf$. 
However, for other categorical compositions we will generally write composition the opposite way, e.g. if we have morphisms \begin{tikzcd}  p \ar[r, "\ell "] & q \ar[r, "k"] & r \end{tikzcd} of a groupoid, then we will write their composition as $\ell k$. 
Roughly, whenever our morphisms represent composition of functions between sets-with-structure, we write their composition as is typical for functional notation; otherwise we write the reverse order.
\subsection*{Organization}
In Section \ref{section:background} we recall the basic definitions of our main objects of study and the monoidal structures involved. 
In Section \ref{section:model-structures} we recall the model structures on our categories of interest.
In Section \ref{section:Quillen-pairs} we define the relevant Quillen adjunctions. 
In section \ref{section:dgbar-analysis} we prove some lemmas about the functor $\dgbarconstruction$ in order to prepare for the main proof, and in Section \ref{section:conservativity} we give the proof of the main theorem.
Finally, Appendix \ref{appendix:the-monoid-axiom} has a technical lemma (that the monoid axiom holds in $\crcom^\times$), and Appendix \ref{appendix:global-strictification-description} gives a description of the functor $\lefttwo$, which we do not use in proving any of our theorems but nevertheless include for completeness. 
\subsection*{Acknowledgments}
This paper is adapted from the third chapter of my PhD thesis, and I would like to thank my advisor Inna Zakharevich for her guidance and encouragement. 
I would also like to thank Chase Vogeli, Varinderjit Mann, and Simon Henry for helpful discussions around the content.
\section{Background}\label{section:background}
\begin{definition}\label{dfn:omega-cats} 
An \emph{$\omega$-category} $C$ is a diagram of sets\footnote{As is typical, we suppress identities from our notation, which would in this diagram be arrows $C_i \to C_{i+1}$.}
\begin{center}
\begin{tikzcd}[sep = huge]
C_0 & C_1 \ar[l, "s", shift left = 2] \ar[l,"t" swap, shift right = 2] & C_2 \ar[l, "s", shift left = 2] \ar[l,"t" swap, shift right = 2]   & \cdots \ar[l, "s", shift left = 2] \ar[l,"t" swap, shift right = 2] 
\end{tikzcd}
\end{center}
such that each diagram
\begin{center}
\begin{tikzcd}[sep = huge]
C_i & C_{i+k} \ar[l, "s^k", shift left = 2] \ar[l,"t^k" swap, shift right = 2] 
\end{tikzcd}
\end{center}
is equipped with the structure of a category with objects $C_i$ and morphisms $C_{i+k}$, and such that these are compatible in the sense that each diagram
\begin{center}
\begin{tikzcd}[sep = huge]
C_i & C_{i+k} \ar[l, "s^k", shift left = 2] \ar[l,"t^k" swap, shift right = 2]  & C_{i+k+j} \ar[l, "s^j", shift left = 2] \ar[l,"t^j" swap, shift right = 2]  
\end{tikzcd}
\end{center}
is a strict $2$-category. This is equivalent to imposing that  
$$(\alpha \circ_j^k \beta) \circ_i^k(\gamma \circ_j^k \eta) = (\alpha \circ_i^k \beta) \circ_j^k(\gamma \circ_i^k \eta)$$
whenever these are defined, where $\circ_a^b$ denotes the categorical composition operation of 
\begin{center}
\begin{tikzcd}[sep = huge]
C_a & C_b \ar[l, "s^{b-a}", shift left = 2] \ar[l,"t^{b-a}" swap, shift right = 2] 
\end{tikzcd}
\end{center} 
This compatibility axiom is referred to as \emph{interchange}. A map between $\omega$-categories is a map of diagrams which preserves all the categorical structure. The resulting category we notate as $\omegacat$.
\end{definition}

\begin{definition}\label{dfn:omega-n-cats}
An \emph{$(\omega,n)$-category} $C$ is an $\omega$-category such that for all $i \ge n$ and $k > 0$, the categories
\begin{center}
\begin{tikzcd}[sep = huge]
C_i & C_{i+k} \ar[l, "s^k", shift left = 2] \ar[l,"t^k" swap, shift right = 2] 
\end{tikzcd}
\end{center}
are groupoids. The resulting full subcategory of $\omegacat$ we notate as $\omegancat{n}$. As a case of particular interest, we refer to $(\omega,0)$-categories as $\omega$-groupoids, and denote the category by $\stinfty$.
\end{definition}
$\omegancat{n}$ possess small limits, and consequentially we can define enriched categories over it:
\begin{definition}
We denote by $\cartesiancat{n}$ the category whose objects are categories enriched in $\omegancat{n}$ with respect to the cartesian monoidal structure, and morphisms functors between such categories. 
\end{definition}
A key feature of the theory is that enriching in $(\omega,n)$-categories ``raises the categorical level by $1$'':
\begin{theorem}\label{thm:enriched-strict-equivalence}
There is an equivalence of categories
	$$\omegannerve{n}:\cartesiancat{n} \rightleftarrows (\omega,n+1)\text{Cat}: \omeganrigidification{n+1}$$ 
\end{theorem}
\begin{proof}
As the result is classical, we sketch just the functor $\omegannerve{n}$ for the point of illustration: for $\C \in \cartesiancat{n}$, we define $\omegannerve{n}(\C)$ by:
$$\omegannerve{n}(\C)_0 = \ob(\C) \quad \text{ and } \quad \omegannerve{n}(\C)_k= \coprod_{x,y \in \ob(\C)} \C[x,y]_{k-1}\text{ for } k > 0$$ 
The source and target maps are given ``locally'' (i.e. by the source and target maps of $\C[x,y]$) for $k > 1$. For $k = 1$ we define $s(\alpha) = x$ where $x$ is the unique object of $\C$ with $\alpha \in \C[x,y]_0$; we define $t$ for $k=1$ similarly. 
The composition operations $\circ_i^k$ are given ``locally'' if $i > 0$; i.e. if $i > 0$ then for $\alpha \circ_i^k \beta$ to be a valid composition in $\omegannerve{n}(\C)$ we must have that $\alpha, \beta \in \C[x,y]$ for some objects $x$ and $y$ and so we can define 
$$\alpha \circ_i^k \beta = \alpha \tilde{\circ}_{i-1}^{k-1} \beta$$
where $\tilde{\circ}$ denotes the composition operation in $\C[x,y]$. 
It is not hard to see that interchange is satisfied for $i < j < k$ when $i > 0$; when $i > 0$ it follows from the composition maps in $C$ being maps of $(\omega,n)$-categories. 
The fact that the appropriate category structures are groupoidal (so that $\omegannerve{n}(\C)$ is not just an $\omega$-category, but an $(\omega,n+1)$-category) follows directly from the fact that the corresponding category structures in $C$ (but shifted down one indexing level) are groupoidal.
\end{proof}
For the results of this paper, we will only need to know that $(\omega,1)\text{Cat}$ is equivalent to $\cartesiancat{0}$. 
Furthermore, we will not actually work directly with $\omega$-groupoids, we will instead work with crossed complexes (Definition \ref{dfn:crossed-complex}). 
\subsection{Crossed Complexes}
\begin{definition}[\cite{Brown_Higgins_Sivera_2011}, Definition 7.1.9]\label{dfn:crossed-complex}
	A \textbf{crossed complex} $C$ is a sequence of sets 
	\begin{center}
	\begin{tikzcd}
		C_0 
		& C_1 
			\ar[l, "s", shift right = 2,swap] 
			\ar[l, "t", shift left = 2] 
		& C_2 
			\ar[l, "\delta_2"] 
		&  C_3 
			\ar[l,"\delta_3"] 
		& \cdots 
			\ar[l]
	\end{tikzcd}
	\end{center}
	Such that:
	\begin{enumerate}
		\item The diagram
		\begin{center}
		\begin{tikzcd}
		C_0 
		& C_1 
			\ar[l, "s", shift right = 2, swap] 
			\ar[l, "t", shift left = 2]
		\end{tikzcd}
		\end{center}
		forms a groupoid, 
		which we will abuse notation by referring to simply as $C_1$.
		\item Each $C_i$ for $i \ge 2$ is a skeletal module 
		over the groupoid $C_1$: that is, a family of groups of the form 
		$$C_i = \coprod_{c \in C_0} C_i(c)$$
		where each $C_i(c)$ is a group, equipped with morphisms
		$$
		\phi_\ell: C_i(s(\ell)) \to C_i(t(\ell))
		$$
		for each $\ell \in C_1$, 
		satisfying that for composable $\ell$ and $p$ in $C_1$, 
		$$
		\phi_{\ell \circ p} = \phi_\ell \circ \phi_p
		$$
		and that $\phi_{\id_x} = \id_{C_i(x)}$. 
		Further, each $C_i(c)$ is abelian for $i>2$. 
		From now on we shall generally suppress the 
		$ c \in C_0$ from our notation when our meaning is clear, 
		saying for example ``$C_i$ is abelian for $i > 2$.''
		\item For $i > 2$, the maps $\delta_i$ are families of maps of groups 
		$\delta_i : C_i \to C_{i-1}$, 
		satisfying $\delta_{i-1} \circ \delta_i = 0$.
		\item $\delta_2$ is a family of maps of groups 
		$\delta_2(c) : C_2(c) \to \Aut(c)$, 
		where by $\Aut(c)$ we mean the automorphism group of $c$ in the groupoid $C_1$.
		\item The action of $C_1$ on $C_i$ is compatible with the 
		$\delta_i$ in the sense that for $i > 2$, 
		$\ell \in C_1$, and $a \in C_i(x)$
		$$\phi_\ell \circ \delta_i = \delta_{i-1} \circ \phi_\ell$$
		\item For any $a \in C_2$, $\delta_2(a)$ acts 
		by conjugation by $a$ on $C_2$ and trivially on $C_i$ for $i > 2$.
	\end{enumerate}
\end{definition}
To a crossed complex $C$, we can associate homology groups:
\begin{definition}
	For $C \in \crcom$ and $c \in C_0$, 
	define $\pi_0$, $\pi_1$, and $H_n$ for $n \ge 2$ by the following:
	\begin{itemize}
		\item $\pi_0(C)$ is $\pi_0$ of the groupoid $C_1$. 
		\item $\pi_1(C,c)$ is $\text{Aut}(C_1,c)/\delta_2(C_2(c))$
		\item For $n \ge 2$, $H_n(C,c)$ is $\text{ker}(\delta_n)(c)/\text{im}(\delta_{n+1})(c)$.
	\end{itemize}	 
	Note that $H_2$ is always abelian although $C_2$ may not be, 
	as the condition that $\delta(C_2)$ acts on itself by conjugation 
	means that $\text{ker}(\delta_2)$ commutes with everything in $C_2$.
\end{definition}
Our interest in crossed complexes comes from the following theorem:
\begin{theorem}[\cite{Brown_1981}]
There is an equivalence of categories
$$ \puttogether : \crcom \rightleftarrows \stinfty : \takeapart$$
\end{theorem}
One should think of a crossed complex as a gadget which encodes a sort of ``basis'' for the cells of the associated $\omega$Gpd; the groups $C_n(x)$ are roughly to be thought of as ``unbased $n$-cells'' in the following sense: let $G \in \stinfty$. For any $n$-cell $\alpha \in G_n$, one can whisker (compose with identity cells) $\alpha$ with the lower dimensional cells $s(\alpha)^{-1}, s^2(\alpha)^{-1},...,s^{n-1}(\alpha)^{-1}$ to obtain an $n$-cell $\tilde{\alpha}$ such that $s(\tilde{\alpha})$ is equal to $\id^{n-1}(p)$ for $p$ the $0$-cell $s^n(\alpha)$. 
For $n \ge 2$, the $n$-cells of this form are a group under the composition operation in 
\begin{center}
\begin{tikzcd}[sep = huge]
C_{n-k} & C_{n} \ar[l, "s^k", shift left = 2] \ar[l,"t^k" swap, shift right = 2] 
\end{tikzcd}
\end{center}
for $ 2 \le k \le n$. Furthermore, by the Eckmann-Hilton argument these groups are abelian if $n \ge 3$.
We think of these $n$-cells $\tilde{\alpha}$ as being the ``unbased $n$-cell data'' of the (based) $n$-cells $\alpha$; given $s(\alpha)$ and $\tilde{\alpha}$ we can recover $\alpha$ via whiskering.
The effect of the functor $\takeapart$ is remember only the cells $\tilde{\alpha}$ and assemble them into a crossed complex, and the effect of $\puttogether$ is to take the cells $\tilde{\alpha}$ and recursively piece them together via whiskering to recover the original $n$-cells. This is analagous to the more well-known Dold-Kan correspondence, which describes how the $n$th term of a chain complex can be thought of as ``unbased $n$-simplex data'' for the corresponding simplicial abelian group, and how these can be recursively pieced together to recover the entire simplicial abelian group.
\\
\indent The reason that we will work with $\crcom$ rather than directly with $\omega \gpd$ is that its resemblance to chain complexes makes it often easier to work with directly than $\stinfty$; in particular in understanding the tensor product. 
We have by the equivalence $ \puttogether : \crcom \rightleftarrows \stinfty : \takeapart$ an equivalence of categories $\omegancat{1} \cong \cartcrossedcat $, and for the remainder of the paper we will use the later as our model of choice for $(\omega,1)$-categories. 
This model will help us elucidate the ``linear'' nature of strict higher categories.
\subsection{Monoidal structures on Crossed Complexes} 
The category of crossed complexes has two monoidal structures that are of interest to us in this paper: the first is the cartesian monoidal structure.
\begin{theorem}
Let $C,D \in \crcom$. Their cartesian product $C \times D$ is the crossed complex
\begin{center}
\begin{tikzcd}
C_0 \times D_0 & C_1 \times D_1 \ar[l, "s", shift right = 2, swap] \ar[l, "t", shift left = 2] & C_2 \times D_2 \ar[l, "\delta_2"] &  C_3 \times D_3 \ar[l,"\delta_3"] & \cdots \ar[l]
\end{tikzcd}
\end{center}
where $C_0 \times D_0$ is the product of sets and $C_i \times D_i$ is the (morphism set of the) cartesian product of groupoids. The boundary maps and the action of $C_1 \times D_1$ are given componentwise, and the natural maps $C \times D \to C$ and $C \times D \to D$ are given by levelwise projections. 
\end{theorem}
The proof is routine. Note that in particular the fact that the action of $C_1 \times D_1$ is given componentwise means that elements in $C_1 \times D_1$ of the form $(c_1, \id_y)$ for $c_1$ an automorphism act trivially on elements of the form $(\id_x, d_n) \in C_n \times D_n$. 
\par The second monoidal structure on $\crcom$ that we are interested in we shall simply call the \textit{tensor product} and notate by $\otimes$. This product corresponds to the Gray tensor product of $\omega$-groupoids along the equivalence of categories $\stinfty \cong \crcom$ (see \cite{Brown_Higgins_Sivera_2011} \cite{Brown2002}). 
Several equivalent definitions are given in the textbook by Brown, Higgins, and Sivera \cite{Brown_Higgins_Sivera_2011}; we give here just the relatively concrete generators-and-relations type definition:
\begin{definition}\label{dfn:tensor-product-of-crossed-complexes}
For $C,D \in \crcom$, their tensor product $C \otimes D$ is presented as follows:
it has generators given by $c_m \otimes d_n \in (C \otimes D)_{m+n}$ for $c_m \in C_m$	and $d_n \in D_n$. The source maps (and target maps, when $m + n = 1$) are given componentwise. 
The relations are given as follows:
$$
(c_m \otimes d_n)^{a_i \otimes b_j} = 
\begin{cases}
c_m^{a_i} \otimes d_n & \text{if }m \ge 2\text{ and } b_j = sd_n \\
c_m \otimes d_n^{b_j} & \text{if } n \ge 2 \text{ and } a_i = sc_m
\end{cases}
$$
$$
(c_mc'_m \otimes d_n) = 
\begin{cases}
c_m \otimes d_n \cdot c'_m\otimes d_n & \text{if }m = 0 \text{ or } n \ge 2 \\
c'_m \otimes d_n \cdot (c_m \otimes d_n)^{c'_m \otimes sd_n} & \text{if } n = 1 \text{ and } m \ge 1
\end{cases}
$$
Symmetrically, we have
$$
(c_m \otimes d_nd'_n) = 
\begin{cases}
c_m \otimes d_n \cdot c_m\otimes d'_n & \text{if } n = 0 \text{ or } m \ge 2 \\
(c_m \otimes d_n)^{sc_m \otimes d_n'} \cdot (c_m \otimes d'_n) & \text{if } n = 1 \text{ and } m \ge 1
\end{cases}
$$
The boundary maps are given by 
$$\del (c_m \otimes d_n) = \del(c_m) \otimes d_n \cdot c_m \otimes \del(d_n)$$
Except if $m$ or $n$ is $1$. If $m = 1$, replace the first term with
$$(sc_1 \otimes d_n)^{-1} \otimes (tc_1 \otimes d_n)^{c_1 \otimes sd_n}$$
and similarly for the case $n = 1$.
If $m = n = 1$, then in particular 
$$\del(c_1 \otimes d_1) = (sc_1 \otimes d_1)^{-1} \cdot (c_1 \otimes td_1)^{-1} \cdot (tc_1 \otimes d_1) \cdot (c_1 \otimes sd_1) $$
\end{definition}
The unit for $\otimes$ is given by the ``one-point crossed complex'' $\bullet$ in which the underlying groupoid is the groupoid with one object and no nonidentity morpisms, and all the higher groups are trivial. 
Since this is also the terminal object, for crossed complexes $C$ and $D$ we have maps $C \otimes D \to C \otimes \bullet \cong C$ and $C \otimes D \to \bullet \otimes D \cong D$; 
these induce a map $C \otimes D \to C \times D$.
\begin{example}\label{example:tensor-v-product-intervals}
let $\mathbb{D}^1$ be the crossed complex with $C_1$ the contractible groupoid \begin{tikzcd} 0 \ar[r, "\ell"] & 1\end{tikzcd} and all higher groupoids trivial, and consider the tensor product $\mathbb{D}^1 \otimes \mathbb{D}^1$. By the description above, the underlying groupoid of $\mathbb{D}^1 \otimes \mathbb{D}^1$ has four generators, which form a \textit{non-commuting} square
\begin{center}
\begin{tikzcd}
0\otimes 0  \ar[r, "\ell \otimes 0"] \ar[d, "0 \otimes \ell"]&  1 \otimes 0 \ar[d, "1 \otimes \ell"] \\
0 \otimes 1 \ar[r, "\ell \otimes 1"] & 1 \otimes 1
\end{tikzcd}
\end{center}
There is one generator in dimension 2, the tensor $\ell \otimes \ell$, and the boundary is exactly the square above (that is, the composition of the morphisms in the square given by starting at $0 \otimes 0$ and then going once counterclockwise around the square). 
We think of this as a ``weakly commuting'' square, with $\ell \otimes \ell$ witnessing the weak commutativity. 
By contrast, the crossed complex $\mathbb{D}^1 \times \mathbb{D}^1$ is the ``strictly commuting'' square, whose underlying groupoid is the \textit{commuting} square
\begin{center}
\begin{tikzcd}[sep = large]
{(0{,} 0)}  \ar[r, "(\ell{,} {\id_0})"] \ar[d, "({\id_0}{,}\ell)"] 
	&  {(1{,}0)} \ar[d, "({\id_1}{,} \ell)"] \\
{(0{,} 1)} \ar[r, "(\ell{,} {\id_1})"] 
	& {(1{,} 1)}
\end{tikzcd}
\end{center}
The natural map $\mathbb{D}^1 \otimes \mathbb{D}^1 \to \mathbb{D}^1 \times \mathbb{D}^1$ is given by collapsing the generating $2$-cell $\ell \otimes \ell$.
\end{example}
This example demonstrates the general behaviour of the map $C \otimes D \to C \times D$, which can be described as collapsing all of the tensors $c_m \otimes d_n$ where both $m$ and $n$ are nonzero. 
It sends a tensor $c_m \otimes d_0$ to the element $(c_m, e)$ where $e$ is the identity element of $D_m(s(d_n))$, and similarly for tensors $c_0 \otimes d_n$. 
This description is fairly immediate from the definition of the map. In Appendix \ref{appendix:global-strictification-description} we give some more description of the behavior of this map.
\subsection{The Model Structure on $\crcom$}
The category $\crcom$ has a model structure originally due to Brown and Golasinski \cite{BrownGolasinski1989}, shown to be equivalent to the folk model structure on $(\omega,0)$-categories by Ara and M\'etayer \cite{Ara_Metayer_11}.  
This model structure is cofibrantly generated, and we will need an understanding of its generating cofibrations and generating acyclic cofibrations to analyze model structures on $\tensorcrossedcat$ and $\cartcrossedcat$.
\begin{definition}
Denote by $\mathbb{G}_n$ the crossed complex which is the image of an $n$-cell under the equivalence of categories $\takeapart: \stinfty \to \crcom$. So, $\mathbb{G}_n$ has two generators in dimensions $0$ through $n-1$, and one generator in dimension $n$. 
\end{definition}
Equivalently, this is the crossed complex obtained as the fundamental crossed complex of the simplicial disk $D^n$ with the hemispherical simplicial structure, which has two nondengerate simplices in dimensions $0$ through $n-1$ and one nondegenerate simplex in dimension $n$. A description of the fundamental crossed complex functor can be found in \cite{Tonks_2003}.
\begin{definition}
For $n \ge 0 $, denote by $\mathbb{S}^{n-1}$ the crossed complex which has a single generator in dimension $1$ and no other generators, and by $\mathbb{D}^n$ the crossed complex which has one generator in dimension $n$ and one generator in dimension $n-1$, and no other generators. For $n = 0$ we declare $\mathbb{S}^{n-1}$ to be empty and $\mathbb{D}^n$ to be a point.
\end{definition}
There is an inclusion $\mathbb{S}^{n-1} \hookrightarrow \mathbb{D}^n$, which adjoins the $n-1$-dimensional generator of $\mathbb{D}^n$. 
For $n = 0$ this is the inclusion of the empty crossed complex into the point, for $n=1$ this is the inclusion of two points into a contractible groupoid on two points, for $n \ge 2$ this is essentially the same as the inclusion of a chain complex on one generator $a$ in dimension $n-1$ into the chain complex with two generators: $a$ in dimension $n-1$ and $b$ in dimension $n$, with $\del b = a$. 
\par From these we get a slightly different description of the model structure on $\crcom$ than either the description given by Brown and Golasinski or the description as the transferred model structure.
We imagine that this description is already known, but we prove it here as we lack a reference.
\begin{theorem}\label{thm:convenient-description-of-model-structure-on-crcom}
There is a cofibrantly generated model structure on $\crcom$ with set of generating cofibrations $I$ the inclusions $i_n: \mathbb{S}^{n-1} \to \mathbb{D}^n$ for $n \ge 0$ and set of generating trivial cofibrations $J$ the inclusions $j_n: \bullet \to \mathbb{D}^n$ for $n \ge 1$. 
\end{theorem}
\begin{proof}
The description of the generating cofibrations is exactly the one given in the original construction \cite{BrownGolasinski1989}, and is also the description of the generating cofibrations transferred along the equivalence of categories $\crcom \cong \stinfty$.
To prove that $J$ gives a set of generating trivial cofibrations, we must prove that each $j_n$ is an acylic cofibration and that they generate a set of generating acylic cofibrations of $\stinfty$ under the equivalence of categories $\crcom \cong \stinfty$. 
The first part is clear from the definitions, so it only remains to show generation.
\par Let $\tilde{J}$ be the set $\{\tilde{j}_n\}_{n \ge 1}$, where $\tilde{j}_n$ is the inclusion $\mathbb{G}_{n-1} \to \mathbb{G}_n$. Then $\tilde{J}$ is a set of generating acylic cofibrations of $\crcom$ (this description of the generating acylic cofibrations is given in \cite{polygraphsbook} section 21.3).
We must show that $J$ generates $\tilde{J}$. 
We do this by noting that the map $\tilde{j}_n: \mathbb{G}_{n-1} \to \mathbb{G}_n$ is itself given as the pushout 
\begin{center}
\begin{tikzcd}
	\bullet \ar[d]  \ar[r, "0"]
		& \mathbb{G}_{n-1} \ar[d, "\tilde{j}_n"] \\
	\mathbb{D}^{n} \ar[r, "f"] 
		& \mathbb{G}_n
\end{tikzcd}
\end{center}
Where the map $f$ sends the top dimensional generator to the top dimensional generator, and thus the $n-1$-dimensional generator to the difference of the two $(n-1)$-dimensional generators of $\mathbb{G}_{n}$. (This is essentially the same argument that a free group on two generators $a$ and $b$ is the same as the free group on $a$ and $ab^{-1}$.) 
\end{proof}
We recall here a classical Quillen adjunction which we will use to induce a Quillen adjunction between model categories of enriched categories.
\begin{theorem}[\cite{Brown_Higgins_1991}]\label{thm:fundamental-crossed-complex-conservative}
There is a conservative (weak-equivalence reflecting) Quillen adjunction 
$$\st_0 : \sset \rightleftarrows \crcom: U_0$$
For a simplicial set $X$ we have $\left(\st_0(X)\right)_0 = X_0$, and for any point $x \in X_0$ we have natural isomorphisms
$$\pi_0(\st_0(X)) \cong \pi_0(X) \quad \pi_1(\st_0(X), x) \cong \pi_1(X,x) \quad \pi_n(\st_0(X),x) \cong \pi_n(\widehat{X}_{x}, x) \text{ for }n> 1$$
Where $\widehat{X}_x$ is a universal cover of $X$ based at $x$. For a crossed complex $C$ we have that $\left(U_0(C)\right)_0 = C_0$, and for any point $c\in C_0$ we natural isomorphisms
$$ \pi_0(U_0(C)) \cong \pi_0(C) \quad \pi_n(U_0(C), c) \cong \pi_n(C,c) \text{ for }n > 0 $$
\end{theorem}
\subsection{Groupoid-equivariant chain complexes}
We recall briefly the basic definitions of groupoid-equivariant chain complexes. For more detail see \cite{Brown_Higgins_Sivera_2011}.
\par 
Let $\chbasic$ be the category of nonnegatively homologically graded chain complexes of abelian groups. 
For $G$ a groupoid, we can consider the category $[G, \chbasic]$ of functors from $G$ to $\chbasic$.
These assemble via a grothendieck construction into
$$\ch := \int_{\gpd} [-, \chbasic]$$
More explicitly, an object of $\ch$ is a pair $(G,C_\bullet)$ where $G$ is a groupoid and $C_\bullet$ is a functor $G \to \chbasic$.  
A morphism $(G,C_\bullet) \to (H, D_\bullet)$ consists of a functor $F:G \to H$, along with a natural transformation 
$\eta: C_\bullet \Rightarrow D_\bullet \circ F $. 
We will frequently refer to an object of $\ch$ as simply $C_\bullet$, suppressing the groupoid part from the notation.
\par 
The category $\ch$ has a monoidal product $\otimes$, which is induced from the monoidal product on $\chbasic$.
The unit for $\otimes$ is the object $\mathbb{Z}[0]_\bullet = (\bullet, \underline{\mathbb{Z}[0]})$, where $\bullet$ is the terminal groupoid and $\underline{\mathbb{Z}[0]} :\bullet \to \chbasic$ is the functor mapping the unique object of $\bullet$ to the chain complex which is $\mathbb{Z}$ in degree $0$ and $\zero$ in every other degree.
We will often be concerned with $\chaug$, the category of \textit{augmented} groupoid-equivariant chain complexes, which is the slice category $\ch /\mathbb{Z}[0]_\bullet$. 
\par 
The other monoidal product on $\ch$ which we will use is the ordinary categorical product, which we will denote $\times$. 
The unit for this product is the terminal object in $\ch$. This is given by the object $(\bullet,\underline{\zero})$ where $\underline{\zero}$ is the constant functor at the zero chain complex. 
\par 
As shown in \cite{Strong2026}, the category $\ch$ has a model structure inherited by viewing it as a Grothendieck construction. It follows that $\chaug$, being a slice category, inherits a model structure. 
It is furthermore shown in \cite{Strong2026} that $\ch$ is a cartesian monoidal model category and that $\chaug$ is a monoidal model category with product induced by $\otimes$: for $C_\bullet,D_\bullet \in \ch$, the product of $C_\bullet \to \mathbb{Z}[0]_\bullet$ and $D_\bullet \to \mathbb{Z}[0]_\bullet$ is given by 
$$C_\bullet \otimes D_\bullet \to \mathbb{Z}[0]_\bullet \otimes \mathbb{Z}[0]_\bullet \to \mathbb{Z}[0]_\bullet$$
Where the second map is the unitor map. 

\section{Model Structures on Enriched Categories}\label{section:model-structures}
The object of this paper overall is to construct and study the strictification functor $\st_1: (\infty,1)\text{Cat} \to (\omega,1)\text{Cat}$, which we will construct explicitly as the composition of left adjoints in the diagram
\begin{center}
\begin{equation}
\begin{tikzcd}[sep = large]
\ssetcat \ar[r, bend left = 20, "\leftone "] & \tensorcrossedcat \ar[l, bend left = 20, "\rightone "] \ar[r, bend left = 20, "\lefttwo "] & \cartcrossedcat \cong \omegancat{1} \ar[l, bend left = 20, "\righttwo "]
\end{tikzcd}
\end{equation}
\end{center}
In this section we describe the model structures on our categories of interest, as well as a few auxiliary categories: the category $\gpdcat$ of categories enriched in (ordinary) groupoids, the category $\trackdgcataug$ of (augmented) categories tensor-enriched in the category $\ch$ of chain complexes over a varying groupoid, and the category $\trackchaincat$ of (unaugmented) categories cartesian-enriched in $\ch$.
The model structure we describe on $\tensorcrossedcat$ is constructed as an application of the theory of Dwyer-Kan model structures developed by Muro \cite{Muro2012}, or alternately the work of Berger and Moerdijk \cite{BergerMoerdijk2012}. Similarly for both $\trackdgcataug$ and $\trackchaincat$, which are detailed in \cite{Strong2026}.
The other model structures (on $\ssetcat$ and $\gpdcat$) are already well-known.
\subsection{The Bergner model structure on $\ssetcat$}
Our model of choice for $(\infty,1)$ categories will be the Bergner model structure on $\ssetcat$. We recall here the definitions: 
\begin{theorem}\cite{BergnerModelStructure}
There is a model category structure on $\ssetcat$ where a morphism $F: \C \to \D$ is:
\begin{enumerate}
	\item A weak equivalence if it induces an equivalence of categories $\pi_0 \C \to \pi_0 \D$, as well as weak equivalences of simplicial sets $\C[x,y] \to \D[Fx,Fy]$ for all $x,y \in \ob \C$.
	\item A fibration if it induces an isofibration of homotopy categories $\Ho(\C) \to \Ho(\D)$, as well as fibrations of simplicial sets $\C[x,y] \to \D[Fx,fF]$ for all $x,y \in \ob \C$.
	\item A cofibration if it has the left lifting property with respect to the acylic fibrations.
\end{enumerate}
\end{theorem}
\subsection{The model structure on $\tensorcrossedcat$.}
Ara and Lucas show in \cite{AraLucasFolkModelMonoidal} that $\omegancat{n}$ is a monoidal model category for the Gray tensor product, and that for $n=0$ this product is symmetric and satisfies the monoid axiom introduced by Schwede and Shipley in \cite{SchwedeShipley1998}. In particular, since the symmetric monoidal structure is equivalent to the tensor product $\otimes$ of crossed complexes under the equivalence of categories between $\stinfty$ and $\crcom$ (see \cite{Brown_Higgins_Sivera_2011} \cite{Brown2002}), $\crcom^\otimes$ is a monoidal model category satisfying the monoid axiom. 
Finally, we note that $\crcom$ is combinatorial, and closed monoidal, 
so overall it is a combinatorial closed symmetric monoidal model category satisfying the monoid axiom.
It then follows by Theorem 1.1 of \cite{Muro2012} that:
\begin{theorem}\label{thm:tensorcrossedcat-dwyer-kan-model-structure-exists}
There is a combinatorial model category structure on $\tensorcrossedcat$ in which a map $F : \C \to \D$ is:
\begin{itemize}
	\item A weak equivalence if each of the maps $\C[x,y] \to \D[Fx,Fy]$ is a weak equivalence in $\crcom$, and also the induced map $\Ho(F): \Ho(\C) \to \Ho(\D)$ is an equivalence of categories, and
	\item an acyclic fibration if it is surjective on objects and each of the maps $\C[x,y] \to \D[Fx,Fy]$ is an acyclic fibration in $\crcom$.
\end{itemize}
\end{theorem}
We will show later using the results of \cite{Strong2026} that the fibrations are precisely the local fibrations which induce isofibrations of homotopy categories (Theorem \ref{thm:tensorcrossedcat-has-naive-fibration-property}).
\subsection{The model structure on $\cartcrossedcat$}
In general, $\omegancat{n}$ has a model structure due to Ara and M{\'e}tayer in \cite{Ara_Metayer_11}, called the \textit{folk model structure}. For our purposes, we only need to consider the model structure on $\omegancat{1}$. 
In this section we will show that $\crcom$ is a cartesian monoidal model category, and that the resulting Dwyer-Kan model structure of $\cartcrossedcat$ coincides with the model structure transferred from the folk model structure on $\omegancat{1}$. 
\begin{theorem}\label{thm:crcom-cartesian-monoidal}
$\crcom$ is a cartesian closed monoidal model category.
\end{theorem}
\begin{proof}
We need to prove three things: closure, the pushout-product axiom, and the unit axiom. 
Since the unit for $\times$ is cofibrant, the unit axiom follows from the pushout-product axiom (see e.g. \cite[Remark A.2]{AraLucasFolkModelMonoidal}). 
Cartesian closure is a special case of \cite[Corollary 2]{Howie1979}. 
So it remains to check the pushout-product axiom.
\par
It suffices to check the generating cofibrations and generating acylic cofibrations. We wish to show that for $i_{m}: \mathbb{S}^{m-1} \hookrightarrow \mathbb{D}^m$, $i_{n}: \mathbb{S}^{n-1} \to \mathbb{D}^n$ generating cofibrations, and $j_k: \bullet \to \mathbb{D}^k$ a generating acyclic cofibration, the dotted map 
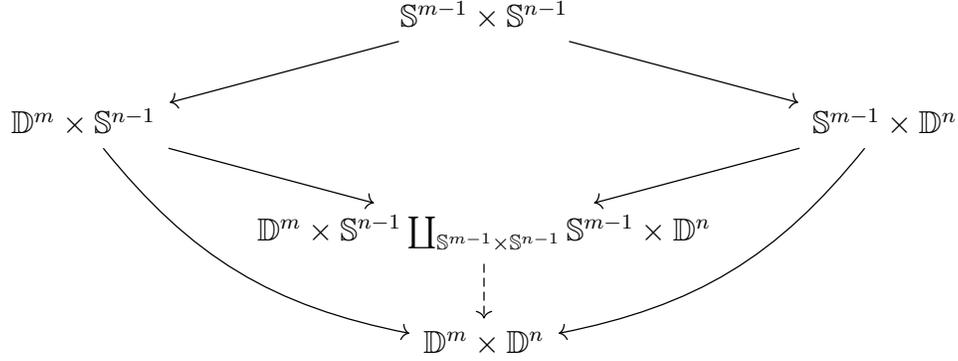
\begin{center}
\begin{tikzcd}
	& \mathbb{S}^{m-1} \times \mathbb{S}^{n-1} \ar[ld] \ar[rd] \\
	\mathbb{D}^m \times \mathbb{S}^{n-1}   \ar[rd] \ar[rdd, bend right = 20]
	&	& \mathbb{S}^{m-1} \times \mathbb{D}^n \ar[ld] \ar[ldd, bend left = 20] \\
	&\mathbb{D}^m \times \mathbb{S}^{n-1}\coprod_{\mathbb{S}^{m-1} \times \mathbb{S}^{n-1}} \mathbb{S}^{m-1} \times \mathbb{D}^n \ar[d, dashed] \\
	& \mathbb{D}^m \times \mathbb{D}^n
\end{tikzcd}
\end{center}
is a cofibration, and similarly the induced map 
$$\mathbb{S}^{m-1} \times \mathbb{D}^k \coprod_{\mathbb{S}^{m-1} \times \bullet} \mathbb{D}^m \times \bullet \to \mathbb{D}^m \times \mathbb{D}^k$$
 is an acyclic cofibration. Working by cases, we can see that in fact for $m,n,k \ge 2$ these are all isomorphisms (essentially because the analagous statement for chain complexes of abelian groups is true); the other cases are not isomorphisms but are equally easy to verify. 
\end{proof}
\begin{theorem}\label{thm:cartcrossed-monoid-axiom}
$\crcom^\times$ satisfies the monoid axiom of \cite{SchwedeShipley2003}.
\end{theorem}
We relegate this proof to an appendix (Appendix \ref{appendix:the-monoid-axiom}), since it requires introducing some additional concepts. Since $\crcom^\times$ is a combinatorial closed symmetric monoidal category satisfying the monoid axiom, we again apply Theorem 1.1 of \cite{Muro2012} to obtain:
\begin{theorem}\label{thm:cartcrossedcat-dwyer-kan-model-structure-exists}
There is a combinatorial model category structure on $\cartcrossedcat$ in which a map $F : \C \to \D$ is:
\begin{itemize}
	\item A weak equivalence if each of the maps $\C[x,y] \to \D[Fx,Fy]$ is a weak equivalence in $\crcom$, and also the induced map $\Ho(F): \Ho(\C) \to \Ho(\D)$ is an equivalence of categories, and
	\item an acyclic fibration if it is surjective on objects and each of the maps $\C[x,y] \to \D[Fx,Fy]$ is an acyclic fibration in $\crcom$.
\end{itemize}
\end{theorem} 
We will show later using the results of \cite{Strong2026} that the fibrations are precisely the local fibrations which induce isofibrations of homotopy categories (Theorem \ref{thm:lefttwo-is-quillen-and-cartcrossed-has-naive-fibration-property}).
This enriched model structure recovers the ``folk'' model structure on $(\omega,1)$-categories.
\begin{theorem}\label{thm:cartcrossedcat-generating-acyclic-cofibrations}
The Dwyer-Kan model structure on $\cartcrossedcat$ and the model structure on $\cartcrossedcat$ induced from the folk model structure of $\omegancat{1}$ are the same. 
\end{theorem}
\begin{proof}
The model structures have the same set of acylic fibrations (which can be checked by checking that the generating cofibrations are the same) and the same fibrant objects (every object is fibrant). 
Thus they are the same model structure by Proposition E.1.10 of \cite{joyalbook}.
\end{proof}
\subsection{The model structure on $\gpdcat$}
 The category $\gpd$ of (ordinary) groupoids is cartesian closed, with the mapping groupoid $\gpd [G,H]$ being the category of functors $G \to H$ (which is a groupoid). 
 Furthermore, it is a closed monoidal model category with respect to the cartesian monoidal structure (the statement for categories is due to Rezk in \cite{Rezk2021}, but the groupoid case follows similarly).
Again applying the machinery of \cite{Muro2012} or \cite{BergerMoerdijk2012}, we get a model structure on $\gpdcat$:
 \begin{theorem}\label{thm:gpd-cat-model-structure-exists}
 	There is a model structure on $\gpdcat$, the category of categories enriched in $\gpd$ with the cartesian monoidal structure, in which a morphism $F: \C \to \D$ is:
 	\begin{itemize}
 		\item A weak equivalence if
 		\begin{itemize}
 			\item each of the maps $\C[x,y] \to \D[Fx,Fy]$ are equivalences of groupoids, and
 			 \item the induced map  $\Ho(F): \Ho(\C) \to \Ho(\D)$ is an equivalence of categories.
		\end{itemize}
 		\item A fibration if it is a local fibration and induces an isofibration of homotopy categories
 	\end{itemize}
 \end{theorem}
This model structure will be useful to us mainly as an auxiliary tool to compare the ``homotopy $(2,1)$-categories'' of our enriched categories (where by $(2,1)$-category we simply mean a $\gpd$-category).
\subsection{The model structure on $\trackdgcataug$ and $\trackchaincat$}
In the companion paper \cite{Strong2026}, it is shown that $\chgpd$ is a closed monoidal model category satisfying the monoid axiom with respect to both $\otimes$ and $\times$.
In particular, the following results are deduced:
\begin{theorem}
There is a model structure on $\trackdgcataug$, the category of categories enriched in $\chaug$ with the monoidal structure given by the tensor product $\otimes$, in which a morphism $F: \C \to \D$ is:
 	\begin{itemize}
 		\item A weak equivalence if
 		\begin{itemize}
 			\item each of the maps $\C[x,y] \to \D[Fx,Fy]$ are equivalences of augmented groupoid-equivariant chain complexes, and
 			\item the induced map  $\Ho(F): \Ho(\C) \to \Ho(\D)$ is an equivalence of categories.
		\end{itemize}
 		\item A fibration if it is a local fibration and induces an isofibration of homotopy categories.
 	\end{itemize}
 \end{theorem}
 
 \begin{theorem}
There is a model structure on $\trackchaincat$, the category of categories enriched in $\chgpd$ with the monoidal structure given by the cartesian product $\times$, in which a morphism $F: \C \to \D$ is:
 	\begin{itemize}
 		\item A weak equivalence if
 		\begin{itemize}
 			\item each of the maps $\C[x,y] \to \D[Fx,Fy]$ are equivalences of groupoid-equivariant chain complexes, and
 			\item the induced map  $\Ho(F): \Ho(\C) \to \Ho(\D)$ is an equivalence of categories.
		\end{itemize}
 		\item A fibration if it is a local fibration and induces an isofibration of homotopy categories.
 	\end{itemize}
 \end{theorem}

\section{The Quillen Adjunctions}\label{section:Quillen-pairs}
In this section we define all of the Quillen adjoint pairs that we discuss towards proving our main theorem.
\par Our Quillen adjunctions will be obtained by applying the results of \cite{Strong2026} to the following commutative (up to natural isomorphism) diagram of Quillen adjunctions between monoidal model categories:
\begin{center}
\begin{equation}\label{diagram:unenriched-functors}
\begin{tikzcd}[sep = huge]
\sset^\times \arrow[r, bend left = 5, shift left = 1, "\st_0 "] 
	& \crcom^\otimes \arrow[r, bend left = 5, shift left = 1, "\Id "] \ar[l, bend left = 5, "U_0 "] \ar[d, bend left = 10, shift left = 2, "\chainification"]
	& \crcom^\times \ar[l, bend left = 5, "\Id "]\ar[d, bend left = 10, shift left = 2, "\chainification"]
\\
	& \chaug^\otimes \ar[u, bend left = 10, "\crossification "]  \ar[r, bend left = 10, "U"]
	& \chgpd^\times \ar[u, bend left = 10, "\crossification "] \ar[l, bend left = 10, "{- \oplus \mathbb{Z}[0]_\bullet}"] \ar[r, bend left = 10, "{\Pi_1}"]
	& \gpd^\times \ar[l, bend left = 10, "{\iota}"]
\end{tikzcd}
\end{equation}
\end{center}
In this diagram, (in which left adjoints go down or rightwards, and right adjoints go up or leftwards) each right adjoint is lax monoidal.
As explained in \cite{Strong2026} or \cite{Muro2012}, since each right adjoint is lax monoidal, we have an induced diagram of adjoints, which we name as in the below diagram:
\begin{center}
\begin{equation}\label{diagram:enriched-functors}
\begin{tikzcd}[sep = huge]
\ssetcat \ar[rr, bend left = 25, "\st_1"] \arrow[r, bend left = 5, shift left = 1, "\leftone "] 
	& \tensorcrossedcat \arrow[r, bend left = 5, shift left = 1, "\lefttwo "] \ar[l, bend left = 10, "\rightone "] \ar[d, bend left = 10, shift left = 2, "\catchainificationaug"]
	& \cartcrossedcat \ar[l, bend left = 10, shift left = 1, "\righttwo "]  \ar[d, bend left = 10, "\catchainification"]
\\
	& \trackdgcataug \ar[u, bend left = 10, "\catcrossificationaug"]  \ar[r, bend left = 10, "\dgbarconstruction"]
	& \trackchaincat \ar[u, bend left = 10, shift left = 1, "\catcrossification"] \ar[l, bend left = 10, "\dgbaradjoint"] \ar[r, bend left = 10, "\hotwo "]
	& \gpdcat \ar[l, bend left = 10, "{\iota}"]
\end{tikzcd}
\end{equation}
\end{center}	
where we have defined $\st_1$ to be the composition $\lefttwo \circ \leftone$ (we omit its adjoint $U_1$ from the diagram).
Since each right adjoint in diagram \eqref{diagram:enriched-functors} is given by applying the corresponding right adjoint in diagram \eqref{diagram:unenriched-functors} locally, it follows that the diagram commutes.
\par
We will describe each adjoint pair in turn and prove that they are Quillen adjunctions, using the results of \cite{Strong2026}.
\subsection{The adjunction between $\ssetcat$ and $\tensorcrossedcat$.}
We describe the ``local strictification'' adjunction 
$$
\leftone: \ssetcat \rightleftarrows \tensorcrossedcat: \rightone
$$
The adjunction is induced from the adjunction $\st_0: \sset \rightleftarrows \crcom :  U_0$ of Theorem \ref{thm:fundamental-crossed-complex-conservative}.
\begin{theorem}
There is a Quillen adjunction 
$$
\leftone: \ssetcat\rightleftarrows \tensorcrossedcat: \rightone
$$
where each functor is the identity on the object set and 
for $\D \in \tensorcrossedcat$ the $\sset$-category $\rightone(\D)$ has $\rightone(\D)[x,y] = U_0(\D[x,y])$. 
\end{theorem}
\begin{proof}
By the main result of \cite{Tonks_2003}, the adjunction $\st_0 \dashv U_0$ is a weak monoidal Quillen pair in the sense of \cite{SchwedeShipley2003}. Thus this follows by Theorem 1.4 of \cite{Muro2012}.
\end{proof} 
Note that while $\rightone$ is given by applying $U_0$ locally, $\leftone$ is not, since $\st_0$ is not strong monoidal.
However, as we stated above, we do know that $\st_0 \dashv U_0$ is weak Quillen monoidal, i.e. in general we have a natural weak equivalence $\st_0(X \times Y) \to \st_0(X) \otimes \st_0(Y)$ induced by the lax monoidal structure of $U_0$.
Hence, by Theorem 1.5 of \cite{Muro2012} we have the following:
\begin{lemma}\label{lem:local-strictification-is-locally-strictification}
Let $\C$ be a cofibrant $\sset$-category. Then for objects $x,y$ in $\C$, we have natural isomorphism in $\Ho(\crcom)$
$$\st_0\left( \C[x,y] \right) \cong \leftone(\C)[x,y]$$
\end{lemma}
Hence, we can compute $\leftone$ up to local weak equivalence (just not up to isomorphism) by applying $\st_0$ to the hom-spaces. 
For example:
\begin{example}
	Consider the $\sset$-category $\C$ given as the pushout in $\ssetcat$
	\begin{center}
	\begin{tikzcd}
	\bullet \ar[d, "\iota_1"] \ar[r, "\iota_0"] & \sus(\Delta^1) \ar[d, dashed]\\
	\sus(\Delta^1) \ar[r, dashed] & \C
	\end{tikzcd}
	\end{center}
	Thus $\C$ has three objects, which we call $\{0,1,2\}$, and $\C[0,1] \cong \C[1,2] \cong \Delta^1$, while $\C[0,2] \cong \Delta^1 \times \Delta^1$, while the other hom-objects are either empty or a point. 
	It is not hard to prove that for any simplicial set $X$, $\leftone(\sus(X)) \cong \sus(\st_0(X))$. So, since $\leftone$ must preserve pushouts, we have a pushout diagram
	\begin{center}
	\begin{tikzcd}
	\bullet \ar[d, "\iota_1"] \ar[r, "\iota_0"] & \sus(\st_0(\Delta^1)) \ar[d, dashed]\\
	\sus(\leftone(\st_0^1)) \ar[r, dashed] & \leftone(\C)
	\end{tikzcd}
	\end{center}
	And so $\leftone(\C)[0,2]$ must be $\Delta^1 \otimes \Delta^1$, which we described in example \ref{example:tensor-v-product-intervals}, in particular it has one generator in dimension $2$.
	On the other one, $\st_1(\C[0,2])$ is the crossed complex which has a fundamental groupoid formed by two noncommuting triangles
	\begin{center}
\begin{tikzcd}[sep = large]
	0\otimes 0  \ar[r] \ar[dr] \ar[d] &  1 \otimes 0 \ar[d] \\
	0 \otimes 1 \ar[r] & 1 \otimes 1
\end{tikzcd}
\end{center}
and two generators in dimension $2$ with boundaries the bottom left triangle and the upper right triangle, corresponding to the two nondegenerate simplices. 
So, $\leftone(\C)[0,2] \not \cong \st_1(\C[0,2])$.
However, they are weakly equivalent, as they are both contractible.
\end{example}
\par Finally, we note the following, which is an application of \cite[Theorem A]{Strong2026}
\begin{theorem}\label{thm:tensorcrossedcat-has-naive-fibration-property}
	The fibrations of $\tensorcrossedcat$ are exactly the local fibrations which induce isofibrations of homotopy categories.
\end{theorem}
\subsection{The adjunction between $\tensorcrossedcat$ and $\cartcrossedcat$}
The adjunction
$$
\lefttwo: \tensorcrossedcat \rightleftarrows \cartcrossedcat : \righttwo
$$
is similar, but less well-behaved than the adjunction $\leftone \dashv \rightone$. 
\begin{theorem}\label{thm:lefttwo-is-quillen-and-cartcrossed-has-naive-fibration-property}
There is a Quillen adjunction 
$$
\lefttwo: \tensorcrossedcat \rightleftarrows \cartcrossedcat : \righttwo
$$
Where each functor is the identity on the object set, and for $\D \in \cartcrossedcat$ we have that $\righttwo(\D)[x,y] = \D[x,y]$. 
Further, the fibrations in $\cartcrossedcat$ are exactly the local fibrations which induce isofibrations of homotopy categories.
\end{theorem}
\begin{proof}
The natural map $C \otimes D \to C \times D$ makes the identity functor on $\crcom$ a lax monoidal functor from $\crcom^\times \to \crcom^\otimes$, in particular the identity adjunction $\Id: \crcom^\otimes \rightleftarrows \crcom^\times : \Id$ has lax monoidal right adjoint. Since the unit $\bullet$ is cofibrant, and the left adjoint preserves the unit, this follows from \cite[Theorem b]{Strong2026}.
\end{proof}
In contrast with the previous adjunction, generally $C \otimes D \not \simeq C \times D$, and therefore we cannot approximate $\lefttwo$ via applying the identity functor to each hom-object.\footnote{The interested reader can see Section \ref{appendix:global-strictification-description} for a somewhat explicit description of the functor $\lefttwo$, though we do will not use this description in any of our arguments.} 

\subsection{The strictification adjunction on $\dg$-track categories}
We recall the following:
\begin{theorem}[\cite{Strong2026}]\label{thm:dgbar-is-quillen}
There exists a Quillen adjunction
\adjunctiondiagram{\dgbarconstruction}{\trackdgcataug}{\trackchaincat}{\dgbaradjoint}
Where each functor is the identity on the object set, and the right adjoint $\dgbaradjoint$ is given by taking the categorical product with the unit for $\otimes$, $\mathbb{Z}[0]_\bullet$.
\end{theorem}
Understanding the left adjoint $\dgbarconstruction$ will be a key part of proving our main theorem. 

\subsection{The functors $\catchainification$ and $\catchainificationaug$}
In \cite[Section 7.4]{Brown_Higgins_Sivera_2011} there is constructed an adjunction
\adjunctiondiagram{\chainification}{\crcom}{\chgpd}{\crossification}
It is straightforward to check that $\chainification$ sends the generating (acyclic) cofibrations of $\crcom$ to (acyclic) cofibrations,
and thus this is a Quillen adjunction. 
It is shown in \cite[Section 9.5.iii]{Brown_Higgins_Sivera_2011} that $\chainification$ is strong monoidal. 
It follows that the left adjoint on slice categories $\chainification_\bullet: \crcom/\bullet \to \chaug$ is also strong monoidal. Since the unit $\bullet$ is also the terminal object of $\crcom$, we therefore obtain:
\begin{theorem}
There exists a Quillen adjunction
\adjunctiondiagram{\catchainificationaug}{\tensorcrossedcat}{\trackdgcataug}{\catcrossificationaug}
In which both functors are identity on object sets, and the left adjoint is given by applying $\chainification_\bullet$ locally.
\end{theorem}
Our cartesian version is more complex: since $\crossification: \chgpd \to \crcom$ is a right adjoint, it is automatically lax monoidal with respect to the cartesian monoidal structures, and hence there is an induced adjunction between categories of (cartesian) enriched categories. It is not immediately obvious, however, that it is Quillen. Fortunately, this is easy to check:
\begin{theorem}
There exists a Quillen adjunction
\adjunctiondiagram{\catfunctor{\chainification}}{\cartcrossedcat}{\trackchaincat}{\catfunctor{\crossification}}
In which both functors are identity on object sets, and the right adjoint is given by applying $\crossification$ locally.
\end{theorem}
\begin{proof}
It suffices to show that $\catcrossification: \trackchaincat \to \cartcrossedcat$ preserves fibrations and acyclic fibrations. Since acyclic fibrations are the local acyclic fibrations which are surjective on objects, and $\catcrossification$ is defined by applying $\crossification$ locally, preservation of acyclic fibrations is immediate. 
By \cite{BergerMoerdijk2012}, the fibrations in each model category are the local fibrations which have the RLP against any single interval in the sense of \cite[Definition 1.11]{BergerMoerdijk2012} (we need only take a single interval because every object in each model category is locally fibrant and hence fibrant).
Preservation of local fibrations is immediate; so it suffices to show that the left adjoint $\catchainification$ sends \textit{a single} interval to an interval. This now follows from commutativity of the diagram \ref{diagram:enriched-functors}, along with the fact that the other adjunctions in the middle square are Quillen.
\end{proof}
We shall need the following fact towards our main theorem.
\begin{theorem}\label{thm:catfunctor-chainification-reflects-weak-equivalences}
$\catchainificationaug: \tensorcrossedcat \to \trackdgcataug$ reflects weak equivalences. 
\end{theorem}
\begin{proof}
As we have noted, $\catchainificationaug$ is given by applying $\chainification$ locally.
Since $\chainification$ is conservative, the result follows. 
\end{proof}

\subsection{The adjunctions with $\gpdcat$}
The simplest adjunction is the lower-right adjunction of Diagram \ref{diagram:enriched-functors}. 
The functor $\Pi_1: \chgpd \to \gpd$ given by forgetting the chain complex part is strong monoidal with respect to the cartesian monoidal structures on both categories, and hence the functor $\hotwo$ is given by applying $\Pi_1$ locally.
\begin{theorem}
The functor $\hotwo: \trackchaincat \to \gpdcat$ is left Quillen. 
\end{theorem}
\begin{proof}
This is again a direct application of Theorem 1.4 of \cite{Muro2012}. 
\end{proof}
We note here that we will abuse notation by also referring to the two functors $\sset \to \gpd$ and $\crcom \to \gpd$ (given by composing left adjoints in Diagram \ref{diagram:unenriched-functors}) as $\Pi_1$, and similarly denoting any compositions of left adjoints in Diagram \ref{diagram:enriched-functors} with target $\gpdcat$ as $\hotwo$. 
We further note that \textit{all} of these functors $\Pi_1$ are strong monoidal with respect to any of the monoidal structures we are considering here, so that \textit{all} of the functors $\hotwo$ are given by applying $\Pi_1$ locally. In each case we think of this as the ``homotopy $(2,1)$-category'' functor. 
Note that what we are calling $\Pi_1: \sset \to \gpd$ (the composition $\sset \to \crcom \to \chgpd \to \gpd$) coincides with the ordinary fundamental groupoid functor;
this reflects the fact that strictification (and the intermediate semi-strictification $(\infty,1)\cat \to \tensorcrossedcat$) preserves not just the homotopy category but the homotopy $(2,1)$-category.

\section{Analysis of the functor $\dgbarconstruction$}\label{section:dgbar-analysis}
In this section we analyze the functor $\dgbarconstruction$. In particular, we will show that it is conservative on certain objects, in order to next section finally conclude with our main theorem: that the strictification of $(\infty,1)$-categories is conservative.
\par
For $\V$ any of our monoidal model categories, recall that the generating cofibrations of $\vcat$ are given by $\sus I$ where $I$ is the set of generating cofibrations of $\V$, along with the inclusion of the empty category into the initial monoid of $\V$ (thought of as a category with one object).
\begin{definition} 
For $\V \in \{\crcom^\otimes, \crcom^\times, \chaug^\otimes, \ch^\times\}$, we shall call a morphism of $\vcat$ \emph{cellular} if it is cellular with respect to the generating cofibrations of $\vcat$. 
We shall call an object cellular if the morphism from the initial object is a cellular morphism.
\end{definition}

\begin{definition}\label{dfn:word-length-cofiltration-category}
Let $\C \in \trackdgcataug$ be cellular. For objects $x,y \in \ob \C$, define a $\Pi_1\C[x,y]$-chain complex by
$E_n \C[x,y] := \{$chains which can be written as a tensor product (composition) of at least $n$ generating chains in dimension $>0\}$.
Define the $\Pi_1 \C[x,y]$ chain complex by $T_n \C[x,y]$ by $\C[x,y]/E_{n+1}[x,y]$, and define the \emph{word length cofiltration} of $\C[x,y]$ by 
\begin{center}
\begin{tikzcd}
		& \cdots \ar[r] 
		& T_n(\C[x,y]) \ar[r] 
		& \cdots \ar[r] 
		& T_2(\C[x,y]) \ar[r]
		& T_1(\C[x,y]) \ar[r] 
		& T_0(\C[x,y]) \cong \mathbf{0}
\end{tikzcd}
\end{center}
Define the \textit{complex of pure length $n$-words} by 
$$P_n\C[x,y] = \ker (T_n\C[x,y] \to T_{n-1}\C[x,y])$$
where this kernel is taken in the category of $\Pi_1 \C[x,y]$ chain complexes. 
\end{definition}
Note that by a dimension argument, the worth length cofiltration converges to $\C[x,y]$.
\begin{definition}\label{dfn:funny-tensor-power}
Fix an indexing set $S$. 
For $n \ge 2$, define a functor $\funnytensor{n}: \trackchaincat_S \to \trackchaingraph_S$ by 
$$\funnytensor{n}(\C)[x,y] = \frac{\coprod_{a_1,...,a_{n-1} \in \ob \C} p_*\left(\C[x,a_1] \otimes \C[a_1,a_2] \otimes \cdots \otimes \C[a_{n-1},y]\right)}{I_\C[x,y]}
$$
where:
\begin{itemize}
\item $p_*$ is the pushforward of groupoid chain complexes along the map of groupoids 
$$[\Pi_1\C[x,a_1] \times \Pi_1\C[a_1,a_2] \times \cdots \times \Pi_1\C[a_{n-1},y] \to \Pi_1\C[x,y]$$
coming from composition in $\C$.
\item $I_\C[x,y]$ is the subcomplex generated by differences
\begin{align*}
&\alpha_1 \otimes \cdots \otimes (\alpha_i,0) \otimes \alpha_{i+1} \otimes \cdots \otimes \alpha_n 
\\
&\; -
\alpha_1 \otimes \cdots \otimes \alpha_i \otimes (0, \alpha_{i+1}) \otimes \cdots \otimes \alpha_n
\end{align*}
\end{itemize}
Where by $(\alpha, 0)$ we mean the image of $\alpha$ along a composition map 
$$
\C[x,y] \hookrightarrow
\C[x,y] \times \C[y,z] \to \C[x,z]$$
where the first map is specified by picking a point in $\C[y,z]$.
Similarly for $(0, \alpha_{i+1})$.
\end{definition}
\begin{theorem}\label{thm:funnytensor-is-pure-tensors}
Let $\D$ be a cellular object of $\tensorcrossedcat$, and set $\C = \catchainificationaug\D$. Then for each $n$ there is a natural transformation 
$$\psi_n : P_n  \C[x,y] \to \funnytensor{n}\dgbarconstruction\left(\C \right)$$
which is an isomorphism modulo the $0$-chains.
\end{theorem}
\begin{proof}
Note that both chain complexes have (up to natural isomorphism) the same groupoid of equivariance, so the groupoid part of natural transformation we can let be this isomorphism.
\par 
$\C[x,y]$ is free (as a chain complex over $\Pi_1\C[x,y]$) on simple tensors, i.e. symbols
$$\alpha_1 \otimes \cdots \otimes \alpha_n$$
where each $\alpha_i$ is a generator adjoined via the description of $\D$ and hence $\C$ as a cellular object.
So it suffices to define $\psi_n$ on the simple tensors. Such an object is of the form:
$$p_1 \otimes \cdots \otimes \alpha_1 \otimes p_i \otimes \cdots \otimes \alpha_2 \otimes p_k \otimes \cdots \otimes \alpha_n \otimes \cdots \otimes p_\ell $$
where $p_1, \dots, p_\ell$ are generating $0$-chains, and $\alpha_1, \dots, \alpha_n$ are generating chains of dimension strictly greater than $0$.
We define $\psi_n$ by sending this to
$$(0, \cdots,  \alpha_1) \otimes (0,\cdots, \alpha_1) \otimes \cdots (0, \cdots, \alpha_n, 0, \cdots, 0) $$
where our notation is as in Definition \ref{dfn:funny-tensor-power}.
This has a left inverse, so it is injective. To show surjectivity after quotienting by the $0$-chains, we need to show that any positive-dimension simple tensor in
$$\coprod_{a_1,...,a_{n-1} \in \ob \C} p_*\left(\C[x,a_1] \otimes \C[a_1,a_2] \otimes \cdots \otimes \C[a_{n-1},y]\right)$$
is in the image of $\psi_n$, up to an element of $I_\C[x,y]$.
\par 
$\dgbarconstruction(\C)[x,y]$ is generated (as a chain complex over $\Pi_1\C[x,y]$) by compositions of generating chains $\alpha_1,\cdots,\alpha_n$ which are all of the same dimension. We write such a composition as $(\alpha_1,\cdots,\alpha_n)$, consistent with our usage of this notation above. 
Now, we have that 
$$(\alpha_1, \cdots , \alpha_n)
= (\alpha_1, 0, \cdots, 0) + (0, \alpha_2, 0,\cdots, 0) + \cdots + (0,\cdots, 0, \alpha_n)$$
It then follows (by bilinearity of the tensor product) that any element of the coproduct above can be written as a sum of elements
$$
\vec{x}_1 \otimes \cdots \otimes \vec{x}_n
$$
Where each $\vec{x}_i$ is a composition $(\alpha_1,\cdots,\alpha_n)$ where exactly one of the entries is nonzero.
Up to the relations imposed by $I_\C[x,y]$, we can assume that each $\vec{x}_i$ is of the form $(0,\cdots,0,\alpha_i)$
except possibly $\vec{x}_n$, which we can only assume is of the form $(0, \cdots,0,\alpha_n, 0, \cdots, 0)$.
Thus, the (degree nonzero) elements in the tensor product $\funnytensor{n}\dgbarconstruction\left(\C \right)$ are sums of elements the form
$$(0, \cdots,  \alpha_1) \otimes (0,\cdots, \alpha_1) \otimes \cdots (0, \cdots, \alpha_n, 0, \cdots, 0) $$
as desired.
\end{proof}
\begin{lemma}\label{lem:funnytensor-homotopical}
Let $S$ be a set, and $f: \C \to \D$ be a cellular map between cellular objects in $\trackchaincat_S$.
If $f$ is a weak equivalence, then so is $\funnytensor{n}(f)$. 
\end{lemma}
\begin{proof}
The main difficulty is proving that $f$ induces a weak equivalence between the subcomplexes $I_\C[x,y] \to I_\D[x,y]$ being quotiented by in the definition of $\funnytensor{n}$. 
\par 
First, note that because every object of $\trackchaincat_S$ is fibrant, and $f$ is a cofibration, there exists $g: \D \to \C$ such that 
$g \circ f = \id_{\C}$ 
and hence by uniqueness of inverses $f \circ g$ is equivalent to $\id_\D$ in the homotopy category of $\trackchaincat_S$. 
Consider the cylinder object for $\D$, $\D \otimes \mathbb{I}_*$, for which 
$$(\D \otimes \I_*)[x,y] = \D[x,y] \otimes \mathbb{I}_*$$
where $\mathbb{I}_*$ is the groupoid chain complex with trivial (terminal) groupoid, and with two generators $p_0$ and $p_1$ in dimension $0$, and one generator $I$ in dimension $1$ with $\del(I) = p_0 - p_1$. 
This yields a well-defined enriched category because tensor product of chain complexes distributes over cartesian product. 
The two inclusions are given by including each chain $\alpha$ as either $\alpha \otimes p_0$ or $\alpha \otimes p_1$. 
\par 
We will show that there is a left homotopy $H: f \circ g \Rightarrow \id_\D$ witnessed by $\D \otimes \I_*$. 
Note that despite bifibrancy of $\D$, this is not automatic because $\D \otimes \I_*$ is not a good cylinder object. 
So, consider the \textit{good} cylinder object $\Cyl(\D)$ given by 
$$\Cyl(\D)[x,y] = \D[x,y] \otimes \I_\to$$
where $\I_\to$ has the same chain complex as $\I$, but the underlying groupoid is the walking isomorphism.
Equivalently, $\Cyl(\D)$ has the same chain complex parts as $\D \otimes \I_*$, but the groupoids of the hom-objects are given by taking the product of the groupoids of the hom-objects of $\D$ with the walking isomorphism. 
Since $\Cyl(\D)$ is a good cylinder object (indeed, the map $\D \coprod \D \to \Cyl(\D)$ is easily seen to be cellular), there is a homotopy $f \circ g \Rightarrow \id_\D$ witnessed by $\tilde{H}:  \Cyl(\D) \to \D$. 
Now consider the diagram 
\begin{center}
\begin{tikzcd}
\D \otimes (\mathbb{Z}[0] \oplus \mathbb{Z}[0])_* 
	\ar[d, hook] \ar[r]
& \Cyl(\D) \ar[d, two heads, "{\sim}" labl, swap]
\\
\D \otimes \I_* \ar[r, two heads, "\sim "]
& \D
\end{tikzcd}
\end{center}
Since $\tilde{H}$ is an acylic fibration and $\D \otimes(\mathbb{Z}[0] \oplus \mathbb{Z}[0])_* \hookrightarrow \D \otimes \I_*$ is cellular (and in particular a cofibration), this diagram has a lift, and the resulting composite
\begin{center}
\begin{tikzcd}
	\D \otimes \I_* \ar[r]
	& \Cyl(\D) \ar[r, "\tilde{H}"] & \D
\end{tikzcd}
\end{center}
gives the desired homotopy witnessed by $\D \otimes \I_*$. Call this homotopy $H$; it yields in particular for any $x$ and $y$ a homotopy $\D[x,y] \to \left(D[x,y]\right)[-1]$ which we also call $H$.
\par 	
This gives us a homotopy $H: \D[x,y] \to \left(D[x,y]\right)[-1]$ between $f \circ g$ and $\id_\D$ to prove that the map $I_\C[x,y] \to I_\D[x,y]$ (definition \ref{dfn:funny-tensor-power}) induced by $f$ is a weak equivalence. 
Again, since $g$ is a left inverse for $f$, we need only construct a chain homotopy between the identity on $I_\D[x,y]$ and 
$f \circ g$ restricted to $I_\D[x,y]$.
We define this chain homotopy $I_\D[x,y] \to \left(I_\D[x,y]\right)[-1]$ explicitly on the generators of $I_\D[x,y]$ via
\begin{align*}
& H(\alpha_1 \otimes \cdots \otimes (\alpha_i,0) \otimes \alpha_{i+1} \otimes \cdots \otimes \alpha_n 
-
\alpha_1 \otimes \cdots \otimes \alpha_i \otimes (0, \alpha_{i+1}) \otimes \cdots \otimes \alpha_n) 
\\
&=
\big(
H(\alpha_1) \otimes \cdots \otimes (\alpha_i,0) \otimes \alpha_{i+1} \otimes \cdots \otimes \alpha_n
\\ 
& \quad \quad 
- H(\alpha_1) \otimes \cdots \otimes \alpha_i \otimes (0, \alpha_{i+1}) \otimes \cdots \otimes \alpha_n 
\big)
\\
&\pm
\big(
f(g(\alpha_1)) \otimes H(\alpha_2) \otimes\cdots \otimes (\alpha_i,0) \otimes \alpha_{i+1} \otimes \cdots \otimes \alpha_n 
\\
& \quad \quad - f(g(\alpha_1)) \otimes H(\alpha_2) \otimes  \cdots \otimes \alpha_i \otimes (0, \alpha_{i+1}) \otimes \cdots \otimes \alpha_n
\big)
\\
&\pm
\cdots 
\\
&\pm
\big(
f(g(\alpha_1)) \otimes \cdots \otimes (f(g(\alpha_i)),0) \otimes f(g(\alpha_{i+1})) \otimes \cdots \otimes H\alpha_n)) 
\\ 
& \quad \quad 
- f(g(\alpha_1)) \otimes\cdots \otimes f(g(\alpha_i)) \otimes (0, f(g(\alpha_{i+1}))) \otimes \cdots \otimes f(g(\alpha_n))
\big)
\end{align*}
(i.e., the standard definition of the chain homotopy induced by tensor product of chain homotopies), where the sign of each term (i.e. the $\pm$) is determined by the sign rule. 
Note that this is a well defined element of $I_{\D}[x,y]$, because the fact that $H$ comes from a map in $\trackchaincat$ means that $H(\alpha_i, 0) = (H({\alpha_i}), 0)$. 
\par 
Thus, it follows that the map $I_\C[x,y] \to I_\D[x,y]$ induced by $f$ is a homology equivalence. The main claim follows by analyzing the map between long exact sequences induced by the map of short exact sequences
\begin{center}
\begin{tikzcd}
I_\C[x,y] \ar[r] \ar[d]
	& 
	\coprod_{a_1,...,a_{n-1} \in J} p_*\C[x,a_1] \otimes \C[a_1,a_2] \otimes \cdots \otimes \C[a_{n-1},y] 
	\ar[r] \ar[d]
	&
	\funnytensor{n}(\C)[x,y]
	\ar[d]
	\\
I_\D[x,y] \ar[r] 
	& 
	\coprod_{a_1,...,a_{n-1} \in J} p_*\D[x,a_1] \otimes \D[a_1,a_2] \otimes \cdots \otimes \D[a_{n-1},y] 
	\ar[r]
	&
	\funnytensor{n}(\D)[x,y]
\end{tikzcd}
\end{center}
We have just shown the left vertical arrow is a homology equivalence, and the middle vertical arrow is a homology equivalence by cofibrancy of the hom-objects and the fact that $\otimes$ and $p_*$ are left Quillen functors.
Thus, the right vertical arrow is a homology equivalence, as desired.
\end{proof}
\begin{theorem}\label{thm:dg-bar-construction-sometimes-conservative}
Fix a set $S$. Let $\C$ and $\D$ be cellular objects of $\trackdgcat_S$ which are in the image of $\left(\catchainificationaug\right)_S$.
Let $f: \C \to \D$ be a cellular map in $\trackdgcataug$ which induces an isomorphism $\hotwo(f): \hotwo(\C) \to \hotwo(\D)$. 
Then if $\dgbarconstruction(f)$ is a weak equivalence, so is $f$.
\end{theorem}
\begin{proof}
Since $\C$ and $\D$ are in the image of $\catchainificationaug$ and $f$ is cellular, the condition that $\hotwo(f)$ is an isomorphism implies that $f$ induces an equivalence of categories. Since it also implies that $f$ induces local equivalences on $\Pi_1$, it suffices to show that $f$ induces local homology equivalences. The map $f: \C \to \D$ induces a map between their word length cofiltrations. 
Since the cofiltrations clearly converge (by a dimension argument), it suffices to show that the induced maps $P_n(\C)[x,y] \to P_n(\D)[x,y]$ are homology equivalences.
The result now follows from Lemma \ref{lem:funnytensor-homotopical} and Theorem \ref{thm:funnytensor-is-pure-tensors}, noting that because $\C$ and $\D$ are in the image of $\catfunctor{\chainification}$ they have isomorphic subcomplexes of $0$-chains, and thus we can quotient by this shared subcomplex without changing whether our induced maps are homology equivalences.
\end{proof}

\section{Proof of conservativity}\label{section:conservativity}
In this section, we will prove our main theorem. First, we will need a couple more technical lemmas:
\begin{lemma}\label{lem:cellular-adjoin-at-chosen-point}
Let $S$ be a set, and $f: \C \to \D$ be a cellular morphism in $\tensorcrossedcat_S$. 
For each connected component of each hom object $\C[x,y]$, make a choice of object $p$. 
Then we may assume that each of the adjoints to the attaching maps of the form $\sus \mathbb{S}^{n} \to \C$ map the unique point of $\mathbb{S}^n$ to one of our choices of points (for $n > 0$).
\end{lemma}
\begin{proof}
Let $f: \mathbb{S}^{n} \to \C[x,y]$ be adjoint to the attaching map.  Let $\bullet$ be the unique point of $\mathbb{S}^n$, and $\ell: p \to \bullet$ any element of $\C[x,y]_1$. 
Then we can define an isomorphism between the pushouts
\begin{center}
\begin{tikzcd}
\mathbb{S}^{n} \ar[r, "f"] \ar[d]
	& \C[x,y] \ar[d]\\
\mathbb{D}^{n+1} \ar[r] 
	& P
\end{tikzcd}
\quad \quad
\begin{tikzcd}
\mathbb{S}^{n} \ar[r, "f^{\ell}"] \ar[d]
	& \C[x,y] \ar[d]\\
\mathbb{D}^{n+1} \ar[r] 
	& P
\end{tikzcd}
\end{center}
Where $f$ is the map which sends the unique point of $\mathbb{S}^1$ to $p$, and the unique $n$-dimensional generator to its action under $\ell$. 
\end{proof}

\begin{lemma}\label{lem:cellular-tensorcrossed-same-points}
Let $S$ be a set, and $f: \C \to \D$ be a morphism in $\tensorcrossedcat_S$ such that:
\begin{itemize}
\item $f$ is cellular with respect to the generating cofibrations of $\tensorcrossedcat_S$, and 
\item $\hotwo(f)$ is a weak equivalence
\end{itemize}
Then there is a $\tilde{\D} \in \tensorcrossedcat_S$ and maps $\tilde{f}: \C \to \tilde{\D}$ and $\iota: \tilde{\D} \to \D$ such that:
\begin{itemize}
\item $\tilde{f}$ is cellular with respect to the set of generating cofibrations of $\tensorcrossedcat_S$, minus the generating cofibration induced by addition of a $0$-cell in a hom-object, and 
\item $\iota$ is cellular with respect to the generating acylic cofibrations which adjoin a $0$-cell and additional $1$-cell. 
\item $f = \tilde{f} \circ \iota$.
\end{itemize}
\end{lemma}
\begin{proof}
By cellularity of $f$, we may write it as an iterated pushout along maps $\sus i_n$ where $i_n: \mathbb{S}^{n-1} \to \mathbb{D}^n$ is a generating cofibration of $\crcom$. Fix some well-ordering $\mathcal{O}$ of the pushouts which adjoin a $1$-cell (i.e. the pushouts along maps $\sus {i_1}$).
 Let $A$ be the set of points adjoined by the pushouts along $\sus i_0$; i.e. $A$ is the set of atomic $0$-cells of the underlying category of $\D$ which are not in $\C$. 
Since $\hotwo(f)$ is a weak equivalence, it is in particular locally $\pi_0$-surjective.
Order $A$ according to the rule $p < q$ if there is some stage of $\mathcal{O}$ in which $p$ is in the image of $\pi_0$ but $q$ is not. 
Rearrange $\mathcal{O}$ so that it consists first of all those elements which affect the above rule, follows by all those which do not.
Then we can replace this first segment of $\mathcal{O}$, along with all of the maps which adjoin the elements of $A$, with acyclic cofibrations adjoining a $0$-cell and $1$-cell, using a technique similar to the proof of Lemma \ref{lem:cellular-adjoin-at-chosen-point}.
\end{proof}

We can now show that the functor $\tensorcrossedcat \to \cartcrossedcat$ reflects weak equivalences between cofibrant objects, from which our main theorem follows.
\begin{theorem}\label{thm:stglobal-conservative}
The functor $\lefttwo: \tensorcrossedcat \to \cartcrossedcat$ reflects weak equivalences between cofibrant objects.
\end{theorem}
\begin{proof}
Firstly, note that if $\lefttwo(f)$ is a weak equivalence, then it is an equivalence on homotopy categories, and therefore so is $f$ (since $\lefttwo$ preserves homotopy categories).
Define a $\C^{\text{Mor}(f)} \in \tensorcrossedcat$ which has the same objects as $\C$ and such that 
$$\C^{\text{Mor}(f)}[x,y] := \D[fx,fy]$$
With composition maps given by the composition maps in $\D$. Then $f$ factors as 
\begin{center}
\begin{tikzcd}
\C \ar[r] & \C^{\text{Mor}(f)} \ar[r, "{\sim}"] & \D
\end{tikzcd}
\end{center}
Where the second map is a weak equivalence. Now, factor the first maps as a cellular map followed by an acylic fibration in $\tensorcrossedcat_{\ob\C}$ to get 
\begin{center}
\begin{tikzcd}
\C \ar[r, hook] & Q\C^{\text{Mor}}(f) \ar[r, "\sim ", two heads] & \C^{\text{Mor}(f)} \ar[r, "{\sim}"] & \D
\end{tikzcd}
\end{center}
Since $\C$ is cofibrant, so is $Q \C^{\text{Mor}(f)}$.
Thus, applying $\lefttwo$, we get 
\begin{center}
\begin{tikzcd}
\lefttwo\C \ar[r, hook] & \lefttwo Q\C^{\text{Mor}}(f) \ar[r, "\sim "] & \lefttwo\D
\end{tikzcd}
\end{center}
We have thus reduced to showing that $\lefttwo$ reflects weak equivalences in the special case that $f$ is identity-on-objects, by the standard small object argument trick we can reduce to the case where the map $f$ is not just a cofibration but cellular with respect to the generating cofibrations of $\tensorcrossedcat_{\ob\C}$. 
\par 
Now we further reduce to the case where $f$ is given by adjoining only cells of dimension $1$ and higher, by Lemma \ref{lem:cellular-tensorcrossed-same-points}.
Hence we can assume that $\hotwo(f)$ is not just a weak equivalence, but an isomorphism. 
By the fact that $\catchainificationaug: \tensorcrossedcat \to \trackdgcataug$ reflects weak equivalences (Theorem \ref{thm:catfunctor-chainification-reflects-weak-equivalences}), it suffices to verify that $\catchainification(f)$ is a weak equivalence. 
This now follows from commutativity of Diagram \ref{diagram:enriched-functors} and Theorem \ref{thm:dg-bar-construction-sometimes-conservative}.
\end{proof}
Now our main result follows readily:
\begin{theorem}[Theorem \ref{alphathm:infinity-one-conservativity}]\label{thm:infinity-one-conservativity-precisely}
The composition $\lefttwo \circ \leftone$ reflects weak equivalences between cofibrant objects.
\end{theorem}
\begin{proof}
We have just proved that $\lefttwo$ reflects weak equivalences between cofibrant objects.
$\leftone$ reflects weak equivalences between cofibrant objects by Lemma \ref{lem:local-strictification-is-locally-strictification} and the fact that $\st_0: \sset \to \crcom$ reflects weak equivalences. 
Since $\leftone$ preserves cofibrant objects, their composition therefore reflects weak equivalences between cofibrant objects.
\end{proof}
\begin{remark}
	Our main theorem can be seen as a direct generalization of the conservativity of strictification of $\infty$-groupoids (Theorem \ref{thm:strictification-infinity-groupoids-conservative}) as follows:
	let $X$ be a simplicial set; then the $\sset$-enriched category $\sus X$ is cofibrant. 
	$\st_1(\sus X)$ is isomorphic to $\sus(\st_0(X))$ by \cite[Lemma 7.3]{Strong2026}.
	Since a map $f:X \to Y$ of simplicial sets is a weak equivalence iff $\sus f: \sus X \to \sus Y$ is a weak equivalence in $\ssetcat$, it follows that for two simplicial sets $X$ and $Y$, $f: X \to Y$ is a weak equivalence iff the map $\st_0(F): \st_0 (X) \to \st_0(Y)$ is a weak equivalence. 
\end{remark}
Of course, we have really proved something slightly stronger than Theorem \ref{thm:infinity-one-conservativity-precisely}:
\begin{theorem}
	The composition of functors $ \catchainification \circ \lefttwo \circ \leftone$ reflects weak equivalences between cofibrant objects.
\end{theorem}
This functor, which takes an $(\infty,1)$-category to a track chain-category, should be thought of as a categorification of the functor which takes a space to its groupoid-equivariant chain complex given by the system of chains on the universal covers, acted on by its fundamental groupoid.

\printbibliography
\appendix
\newpage
\section{The monoid axiom in $\crcom^\times$}\label{appendix:the-monoid-axiom}
In this appendix we will prove that the monoid axiom of \cite{SchwedeShipley1998} holds for the monoidal model category $\crcom^\times$.
The proof will proceed with the same techniques as in \cite{AraLucasFolkModelMonoidal} in which the analagous statement is proved for the Gray tensor product of $\omega$-categories. Of course, our case is easier to describe, since $\omega$-groupoids (and thus, crossed complexes) are simpler objects than general $\omega$-categories. 
Some of the results in this section could be proved by using the lax structure of the functor $\Id: \crcom^\times \to \crcom^\otimes$ and appealing to results in \cite{AraLucasFolkModelMonoidal}, but we prove them directly, as they are all short and the independent proofs are (we think) of some interest, as they reflect doing enriched category theory in the $(\omega,1)$-category of $\omega$-groupoids, as opposed to the $\crcom^\otimes$-category of $\omega$-groupoids. 
\par 
Throughout this section, let $\iso$ denote the crossed complex with underlying groupoid consisting of two objects $0$ and $1$ and a single nonidentity isomorphism, and all higher groups trivial. This is the same as what we have called $\mathbb{D}^1$ in our description of the model structure on $\crcom$ (Theorem \ref{thm:convenient-description-of-model-structure-on-crcom}), but we change notation in this section to match better with the notation in \cite{AraLucasFolkModelMonoidal}. We will denote by $\pi: \iso \to \bullet$ the unique map to the terminal object. We will denote by $\partial \iso$ the ``boundary'' of $\iso$; i.e. the crossed complex $\bullet \coprod \bullet$, which has an inclusion given by the two objects of $\iso$, $0$ and $1$.
\begin{definition}
Let $f,g: X \to Y$ be maps of crossed complexes. A \emph{$\iso$-tranformation} from $f$ to $g$, written $h: f \Rightarrow g$, is a map of crossed complexes $h: \iso \times X \to Y$ making the diagram 
\begin{center}
\begin{tikzcd}
X  \ar[rd, "0 \times \id_X", swap] \ar[rrd, "f", bend left = 20]
\\
& \iso \times X \ar[r, "h"]
	& Y 
\\
X \ar[ru, "1 \times \id_X"] \ar[rru, "g", bend right = 20, swap]
\end{tikzcd}
\end{center}
commute. We write $h: f \Rightarrow g$ to denote that $h$ is a $\iso$-transformation from $f$ to $g$.
\end{definition}
\begin{lemma}\label{lem:iso-transforms-induce-isos-of-maps}
If $f,g: X \to Y$ are maps of crossed complexes, and there exists an $\iso$-transformation $h: f \Rightarrow g$, then $f$ and $g$ induce isomorphic maps on $\pi_n$, $n \ge 0$.
\end{lemma}
\begin{proof}
Let $\ell$ denote the nonidentity isomorphism of $\iso$. 
For any $x_0 \in X$ and any $n \ge 1$, define a map $\psi_\ell: \pi_n(Y,f(x_0)) \to \pi_n(Y, g(x_0))$ by 
$$\psi_\ell(\alpha) = \phi_{h(\id_{x_0}, \ell)}(\alpha)$$
where $\phi_{h(\ell)}$ is the action of the groupoid $Y_1$ on the groups $Y_n$ (for the case $n=1$, this is conjugation). Since this action commutes with the differentials, this map is well defined on the homotopy groups. 
Note that for $(\alpha,\id_0) \in (X \times I)_n((x_0,0))$, 
$$\phi_{(\id_{x_0}, \ell)}(\alpha,\id_0) = (\alpha, \id_1)$$ 
From this it follows that for $n\ge 1$ we have a commutative triangle
\begin{center}
\begin{tikzcd}
& & \pi_n(Y,f(x_0)) \ar[dd,  "\sim " {rotate=90, anchor=south}, "\psi_\ell "] \\
\pi_n(X,x_0) \ar[rru, "f_*"] \ar[rrd, "g_*", swap] \\
& & \pi_n(Y,f(x_0))
\end{tikzcd}
\end{center}
as desired. The case $n=0$ follows similarly, or as a special case where we consider $C_1$ to act on $C_0$ via $\phi_\ell(s(\ell)) = t(\ell)$.
\end{proof}
\begin{definition}\label{dfn:strong-iso-retract}
Let $i: X \to Y$ be a map of crossed complexes. We say $i$ is a \emph{$\iso$-transformation retract} if it admits a retraction $r:Y \to X$, along with a $\iso$-transformation $h: i r \Rightarrow \id_Y$. 
We furthermore say that $i$ is a \emph{strong $\iso$-transformation retract} if $r$ and $\alpha$ can be chosen such that the diagram
\begin{center}
\begin{tikzcd}
\iso \times X \ar[r, "\pi \times X "] \ar[d, "\iso \times i"]
	& X \ar[d, "i"]
\\
\iso \times Y \ar[r, "h"]
	&  Y
\end{tikzcd}
\end{center}
commutes.
\end{definition}
\begin{lemma}\label{lem:every-iso-retract-equivalence}
A $\iso$-transformation retract is a weak equivalence.
\end{lemma}
\begin{proof}
This follows from \ref{lem:iso-transforms-induce-isos-of-maps}.
\end{proof}
\begin{lemma}\label{lem:every-acylic-cof-iso-retract}
Every acyclic cofibration in $\crcom$ is a strong $\iso$-transformation retract.
\end{lemma}
\begin{proof}
Let $i:X \to Y$ be an acylic cofibration. Since $X$ is fibrant, we have a retraction $r: Y \to X$.
Since $\crcom^\times$ is a monoidal model category (Theorem \ref{thm:crcom-cartesian-monoidal}), the pushout product map $\iso \times X \coprod_{\partial \iso \times X}\partial \iso \times Y \to \iso \times Y$ is an acyclic cofibration. Hence we have a lift
\begin{center}
\begin{tikzcd}[sep = huge]
\iso \times X \coprod_{\partial \iso \times X}\partial \iso \times Y \ar[r, "(\pi \times i{,} \id_y \amalg ri)"] \ar[d]
	& Y 
\\
\iso \times Y \ar[ru, dashed, "h"]
\end{tikzcd}
\end{center}
And this lift gives us precisely the data to demonstrate $i$ as a strong $\iso$-transformation retract.
\end{proof}
\begin{lemma}\label{lem:iso-retract-pushout-stability}
Strong $\iso$-transformation retracts are stable under pushouts. 
\end{lemma}
\begin{proof}
Let $i:X \to Y$ be a strong $\iso$-transformation retract, and $f: X \to X'$ a map of crossed complexes. We wish to prove that in the pushout diagram
\begin{center}
\begin{tikzcd}
X \ar[d, "i"] \ar[r, "f"]
	& X' \ar[d, dashed, "i'"] \\
Y \ar[r, dashed, "g"]
	& Y'
\end{tikzcd}
\end{center}
the arrow $i': X' \to Y'$ is a strong $\iso$-transformation retract. By assumption, we have maps $r: Y \to X$ and $h: Y\times I \to Y$ satisfying the conditions of Definition \ref{dfn:strong-iso-retract}. Consider the following commutative diagrams:
\[
\begin{tikzcd}
X \ar[dd, "i"] \ar[rr, "f"] \ar[dddd, "\id_X", bend right = 25, shift right = 2, swap]
	&& X' \ar[dd, "i'"] \ar[dddd, "\id_{X'}", bend left = 25, shift left = 2] \\\\
Y\ar[dd, "r"] \ar[rr, "g"]
	&& Y' \ar[dd, dashed, "r'"] \\\\
X \ar[rr, "f"] 
	&& X' 
\end{tikzcd}
\qquad
\begin{tikzcd}
& \iso \times X \ar[ldd, "\pi \times X", bend right = 30, swap] \ar[dd, "\iso \times i"] \ar[rr, "J_1 \times f"] 
	& &  \iso \times X' \ar[dd, "\iso \times i'"] \ar[rdd, "\pi \times X'", bend left = 30]  \\\\
X \ar[rdd, "i", bend right = 30, swap]
	& \iso \times Y\ar[dd, "h"] \ar[rr, "J_1 \times g"]
	&& \iso \times Y' \ar[dd, dashed, "h'"]
	& X' \ar[ldd, "i'", bend left = 30] \\\\
&	Y \ar[rr, "f"] 
	&& Y 
\end{tikzcd}
\]
We have used the universal property of the pushout to obtain the maps $r'$ and $h'$. From the commutativity of the diagrams, we verify that $r'$ and $h'$ together give the data to demonstrate that $i'$ is a strong $\iso$-transformation retract, as desired. 
\end{proof}
We can now prove that $\crcom^\times$ satisfies the monoid axiom:
\begin{proof}[Proof of Theorem \ref{thm:cartcrossed-monoid-axiom}]
We wish to prove that the set of transfinite compositions of pushouts with respect to the generating set
$$
\{X \times j \; | \; X \in \crcom \text{ and } j \text{ is an acylic cofibration of } \crcom\}	
$$
consists of weak equivalences. Since every such $j$ is by Lemma \ref{lem:every-acylic-cof-iso-retract} a strong $\iso$-transformation retract, and strong $\iso$-transformation retracts are (from the structure of the definition) stable under taking products with objects, it suffices to prove that transfinite compositions of pushouts with respect to a set of $\iso$-transformation retracts are weak equivalences. 
The statement on pushouts is Lemma \ref{lem:iso-retract-pushout-stability} (along with \ref{lem:every-iso-retract-equivalence}), so the result follows from the fact that weak equivalences are closed under transfinite composition.  
\end{proof}

\section{The Functor $\lefttwo$}\label{appendix:global-strictification-description}
In this appendix we give a (somewhat) explicit description of the functor $\lefttwo$. We have not needed to use any results from this section, and so our treatment is somewhat informal---our goal is not to prove theorems, but only to give the interested reader an understanding of $\lefttwo$. In particular, we wish to motivate calling it a ``linearization.'' 
First we give an example, which will motivate the general case.
\begin{example}\label{example:global-strictification-pushout}
Recall that $\mathbb{D}^1$ is the crossed complex whose underlying groupoid has two objects and one nonidentity isomorphism, and all higher groups are trivial. 
Consider the $\crcom^\otimes$-category $\C$ given as the following pushout in $\tensorcrossedcat$:
\begin{center}
\begin{tikzcd}
\bullet \ar[d, "\iota_0"] \ar[r, "\iota_1"] & \sus(\mathbb{D}^1) \ar[d, dashed]\\
\sus(\mathbb{D}^1) \ar[r, dashed] & \P_{1,1}
\end{tikzcd}
\end{center}
Thus $\P_{1,1}$ has three objects, which we call $\{0,1,2\}$, and $\P_{1,1}[0,1] \cong \P_{1,1}[1,2] \cong \mathbb{D}^1$, while $\P_{1,1}[0,2] \cong \mathbb{D}^1 \otimes \mathbb{D}^1$, while the other hom-objects are either empty or a point. 
It is not hard to prove (see \cite{Strong2026})2w  that for any crossed complex $C$, $\lefttwo(\sus^\otimes(C)) \cong \sus^\times(C)$. So, since $\lefttwo$ must preserve pushouts, we have a pushout diagram
\begin{center}
\begin{tikzcd}
\bullet \ar[d, "\iota_0"] \ar[r, "\iota_1"] & \sus(\mathbb{D}^1) \ar[d, dashed]\\
\sus(\mathbb{D}^1) \ar[r, dashed] & \lefttwo(\P_{1,1})
\end{tikzcd}
\end{center}
And so $\lefttwo(\P_{1,1})[0,2]$ must be $\mathbb{D}^1 \times \mathbb{D}^1$.
The unit map $\P_{1,1} \to \righttwo \circ \lefttwo (\P_{1,1})$ is, restricting to the hom-object $\P_{1,1}[0,2]$, the natural map $\mathbb{D}^1 \otimes \mathbb{D}^1 \to \mathbb{D}^1 \times \mathbb{D}^1$, which we described in Example \ref{example:tensor-v-product-intervals}. 
\end{example}
It follows from the description in Example \ref{example:tensor-v-product-intervals} of the natural map $\mathbb{D}^1 \otimes \mathbb{D}^1 \to \mathbb{D}^1 \times \mathbb{D}^1$ as quotienting out the tensor $\ell \otimes \ell$ that for any arbitrary $\C \in \tensorcrossedcat$, the unit map $\C \to \righttwo \circ \lefttwo \C$ must ``quotient out'' any $\alpha \in \C[x,y]_2$ which can be written as a composition of two $1$-cells. Similar statements could be made about compositions of $m$-cells and $n$-cells for $m,n \neq 0$, and in the rest of this section we will show that this behaviour completely describes $\lefttwo$.
\par
Recall that for crossed complexes $C$ and $D$ we have a map $C \otimes D \to C \times D$, described on simple tensors by
$$
a_m \otimes b_n \mapsto 
\begin{cases}
0 & \text{ if } m > 0 \text{ and } n > 0 \\
(\id_{a_m},b_n) & \text{ if } m = 0, n > 0 \\
(a_m, \id_{b_n}) & \text{ if } m >0, n =0 \\
(a_m,b_n) & \text{ if } m=n=0
\end{cases}
$$
From this it is easy to see that this map will be surjective. Further, from this description we will see that the degreewise kernel is generated by tensors $a_m \otimes b_n$ where $m$ and $n$ are strictly positive, along with boundaries of such elements. First, let us clarify what we mean by ``kernel:"
\begin{definition}
For $G$ and $H$ groupoids on the same object set, and $f: G \to H$ a map of groupoids which is identity-on-objects, the \emph{kernel} of $f$, $\ker(f)$, is the subgroupoid of $G$ consisting of all objects of $G$ along with all morphisms which are sent to identities by $f$.
\end{definition} 
As with other kernels, for any surjective morphism of groupoids $f: G \to H$ which is identity-on-objects, we have a pushout square
\begin{center}
\begin{tikzcd}
\ker(f) \ar[r] \ar[d]
& G \ar[d, "f"] \\
G_0 \ar[r] 
& H
\end{tikzcd}
\end{center} 
where $G_0$ is the trivial groupoid on the object set $G_0 = H_0$. Since $C \otimes D$ and $C\times D$ have the same object set, $C_0 \times D_0$, we can describe the surjective map $C \otimes D \to C \times D$ in terms of kernels of this kind.
\subsection*{Description of the kernel of $\mathbf{C \otimes D \to C \times D}$}
For crossed complexes $C$ and $D$, the degreewise kernels of the natural map $C \otimes D \to C \times D$ can be described as follows:
\begin{itemize}
\item The kernel of $(C \otimes D)_1 \to (C\times D)_1$ is generated by boundaries of tensors $c_1 \otimes d_1 \in (C\otimes D)_2$.
\item The kernel of $p_*(C \otimes D)_2 \to (C \times D)_2$ is generated as a (crossed) module over $C_1 \times D_1$ by tensors $c_1 \otimes d_1$ as well as boundaries of elements $c_i \otimes d_j \in (C \otimes D)_3$ where both $i$ and $j$ are nonzero. Here by $p_*$ we mean the pushforward (or ``induced'') crossed module along the map $p: (C\otimes D)_1 \to (C\times D)_1$.
\item For $n \ge 3$, the kernel of $(C \otimes D)_n \to (C \times D)_n$ is generated as a module over $\Pi_1(C\otimes D) \cong \Pi_1(C\times D)$ by tensors $c_i \otimes d_j$ where $i+j = n$ and $i,j \neq 0$, as well as boundaries of such elements.
\end{itemize}
The degree $1$ description follows from the description of $(C \otimes D)_1$ as a pushout in \cite{Brown_Higgins_Sivera_2011}. 
The degree $2$ description similarly follows from the description there of the crossed module part of $C\otimes D$ as a coproduct. 
The higher $n$ are simpler: for $n \ge 3$, $(C\otimes D)_n$ and $(C \times D)_n$ are modules over the groupoid $\Pi_1(C \otimes D) \cong \Pi_1(C \times D)$, so this follows from the description of the map on generators, along with the observation that $(C\times D)_n$ is isomorphic to the direct sum of the modules generated by simple tensors $c_0 \otimes d_n$ and $c_m \otimes d_0$.
\par
For the rest of this section, we will go about defining a functor $\tencart{(-)}: \tensorcrossedcat \to \cartcrossedcat$, and then we will show that it is left adjoint to $\righttwo$. 
\subsection*{Description of the functor on objects:} We define $\tencart{\C}$ to have the same objects in $\C$, and hom objects as follows: for $a,c \in \ob(\C)$, we let 
$$
\tencart{\C}[a,c]_0 = \C[a,b]_0
$$
$$\tencart{\C}[a,b]_i = \C[a,b]_i/K$$
Where $K$ is the subgroupoid of $\C[a,b]_i$ generated by the union over all tuples $r_1,...,r_n \in \ob(C)$ of the image of the kernel of the maps $\C[a,r_1] \otimes \cdots \otimes \C[r_n,b] \to \C[a,r_1] \times \cdots \times \C[r_n,b]$ under the composition map $\circ : \C[a,r_1] \otimes \cdots \otimes \C[r_n,b] \to \C[a,b]$. 
In other words, $K$ is generated by compositions $\circ(x_1 \otimes \cdots \otimes x_n)$ where each $x_n$ is in degree strictly positive; along with boundaries of such elements.
This describes the objects and hom-objects of $\tencart{\C}$, but not the composition operation. To define the composition operation, we must define maps 
$$\tencart{\circ}: \tencart{\C}[a,b] \times \tencart{\C}[b,c] \to \tencart{\C}[a,c]$$
On objects, we can let this be the same as $\circ$. On higher cells, we define $\tencart{\circ}([a],[b]) = [\circ(a\otimes s(b) \cdot s(a) \otimes b)]$, where $s$ denotes the source object of a morphism in a groupoid.
We must prove that this is independent of the choice of representatives $a$ and $b$. Let us suppose that $[a] = [a']$. 
Then $a$ and $a'$ differ by some combination of elements in the images of kernels of maps $\C[a,r_1] \otimes \cdots \otimes \C[r_n,b] \to \C[a,r_1] \times \cdots \times \C[r_n,b]$. We write
$$a^{-1}a' = k_1 \cdots k_m$$
and then note that this gives us that 
\begin{align*}
\circ(a \otimes s(b))^{-1} \circ(a' \otimes s(b)) &= \circ(a^{-1} a' \otimes s(b)) \\
&= \circ(k_1 \cdots k_m \otimes s(b)) \\
&= \circ(k_1 \otimes s(b)) \cdots \circ (k_m \otimes s(b)) \\
\end{align*}
By our description of the kernels in terms of the degrees, we have that $\pi(k_i \otimes s(b)) = 0$, and hence by definition $\circ(k_i \otimes s(b))$ is $0$ in $\tencart{\C}[a,c]$, as desired. We can symmetrically argue that this is independent of the choice of representative $b$. 
\subsection*{Description of the functor on morphisms:} Given a morphism $F: \C \to \D$ in $\tensorcrossedcat$, we define $\tencart{F}$ to be the same as $F$ on objects, and on hom-objects induced by the description of the hom-objects of $\tencart{\C}$ and $\tencart{\D}$ as quotients. 
Here we are using that since $F$ is a functor between enriched categories, it preserves compositions, and so it maps the kernel of the map $\C[a,b]_i \to \tencart{\C}[a,b]_i$ into the kernel of the map $\D[Fa,Fb]_i \to \tencart{\D}[Fa,Fb]_i$. 
\subsection*{Proof that this defines an adjoint:}
Note that the functor $\righttwo$ is fully faithful, which is fairly straightforward from its definition as the identity on hom-objects. 
So to show that $(\tencart{-})$ defines an adjoint to $\righttwo$, it suffices to show that for $\C \in \tensorcrossedcat$, the map $\C \to \tencart{\C}$ is initial among maps $\C \to \righttwo(\D)$. 
Since this map is a quotient on hom-objects, it suffices to show that the elements we quotient by are sent to identities in $\righttwo(\D)$, and that the induced maps on hom-objects assemble to form a functor of enriched categories; i.e. they commute with composition. The first statement follows immediately from the descriptions in terms of dimensions. The second statement follows from observing that the map $C \otimes D \to C \times D$ sends $c_m \otimes d_0 \cdot c_0 \otimes d_n$ to $(c_m, \id_{d_0})\cdot (\id_{c_0},d_n) =(c_m,d_n)$.

\end{document}